\documentclass[a4paper,12pt]{article}

\usepackage{amsthm}
\usepackage{amsmath}
\usepackage{epsfig}
\usepackage{amssymb}
\usepackage{latexsym}
\usepackage{longtable}
\usepackage[T1]{fontenc}
\usepackage[utf8]{inputenc}
\usepackage{color}
\usepackage{enumitem}

\author{
Robert Lukoťka\\Comenius University, Bratislava\\{\small\tt lukotka\@@dcs.fmph.uniba.sk}
}
\title{Short cycle covers of cubic graphs and intersecting $5$-circuits}

\theoremstyle{definition}
\newtheorem{definition}{Definition}
\newtheorem{theorem}[definition]{Theorem}

\newtheorem{lemma}[definition]{Lemma}
\newtheorem{conjecture}{Conjecture}[section]

\newcommand{\C}{{\cal C}}

\newcommand{\Z}{\mathbb{Z}}

\newcommand{\cR}{{\text{R}}}
\newcommand{\cG}{{\text{G}}}
\newcommand{\cB}{{\text{B}}}

\newcommand\myto[1]{$\stackrel{\text{#1}}{\to}$}

\let\epsilon=\varepsilon

\let\phi = \varphi

\begin{document}

\maketitle


\abstract{  
A \emph{cycle cover} of a graph is a collection of cycles such that each edge of the graph is contained in at least one of the cycles. The length of a cycle cover is the sum of all cycle lengths in the cover.
We prove that every bridgeless cubic graph with $m$ edges has a cycle cover of length at most $212/135 \cdot m \ (\approx 1.570 m)$. Moreover, if the graph is cyclically $4$-edge-connected we obtain a cover of length at most
$47/30 \cdot m \approx 1.567 m$.
}


\section{Introduction}

A \emph{cycle} is a graph each vertex of which has an even degree. A \emph{circuit} is a $2$-regular connected graph. A \emph{cycle cover} of a graph $G$ is a multiset of cycles from $G$ such that each edge of $G$ is contained in at least one of the cycles. As every cycle is decomposable into circuits, a circuit cover can be obtained in a straightforward manner from a cycle cover by decomposing the cycles into circuits. The \emph{length of a cycle} is the number of its edges, the \emph{length of a cycle cover} is the sum of the lengths of all cycles in the cover. If $G$ has a bridge (a \emph{bridge} is an edge whose removal increases the number of components of the graph), then the bridge does not belong to any cycle of $G$. Thus only bridgeless graphs are of interest regarding cycle covers. In this paper we are interested in finding cycle covers with small length.

Alon and Tarsi \cite{AT}, and independently Jaeger \cite{raspaud} conjectured the following.
\begin{conjecture}[Short Cycle Cover Conjecture]
\label{sccc}
Every bridgeless graph with $m$ edges has a cycle cover of length at most $1.4m$.
\end{conjecture}
A shortest cycle cover of the Petersen graph has length $21$ and consists of four cycles. 
The upper bound given by Conjecture \ref{sccc} is,  therefore, tight and the constant $1.4$ cannot be improved.
The Short Cycle Cover Conjecture is surprisingly strong as, among others, it implies \cite{JT}
the famous Cycle Double Cover Conjecture by Szekeres \cite{szekeres} and Seymour \cite{seymour}.

The best general result on short cycle covers is due to Alon and Tarsi \cite{AT} and Bermond, Jackson, and Jaeger \cite{BJJ} who independently proved that every bridgeless graph with $m$ edges has a cycle cover of length at most $5/3 \cdot m$. Despite numerous results on short cycle covers for special classes of graphs (see e.g. Chapter 8 of the book \cite{CQbook}) no progress was made for the general case. 
Significant attention has been devoted to cubic graphs and graphs with minimal degree three. 

Let $G$ be a bridgeless graph on $m$ edges.
Among the recent results,
 Kaiser et.~al.~\cite{kaiser} proved that $G$ has a cycle cover of length at most
$34/21 \cdot m\approx 1.619m$ for $G$ cubic and 
$44/27 \cdot m \approx 1.630m$ for $G$ loopless without vertices of degree two.
Fan \cite{fan} proved that $G$ has a cycle cover of length at most
$29/18 \cdot m \approx 1.611m$ for $G$ cubic,
$218/135 \cdot m\approx 1.611m$ for $G$ loopless without vertices of degree two, and
$278/171 \cdot m \approx 1.626m$ for $G$ without vertices of degree two.
Candr\'akov\'a and Luko\v tka \cite{sccc} proved 
the bound $1.6m$ for $G$ cubic,
$212/135 \cdot m \approx 1.570 m$ for $G$ cubic without intersecting $5$-circuits, and 
$14/9 \cdot m \approx 1.556 m$ for $G$ cubic without $5$-circuits.

In this paper we show how to overcome the problems in the proof from \cite{sccc} 
arising from $5$-circuits that intersect and obtain the following result.
\begin{theorem}\label{thm:main}
Let $G$ be a bridgeless cubic graph on $m$ edges. Then $G$ has a cycle cover consisting of three cycles of length at most $212/135 \cdot m \approx 1.570 m$.
\end{theorem}
\noindent This result makes the biggest improvement of the bound on the length of shortest cycle cover of cubic graphs made by a single paper up to date (see e.g. \cite{sccc} for previous best bounds) and the proof is simpler than in \cite{sccc}.
Besides the idea concerning intersecting $5$-circuits, 
the structure of covers follows that of \cite{kaiser} while we use several optimisations 
presented in \cite{sccc}.  The ideas of Fan \cite{fan} make many steps in our proof easier, which was crucial for us to examine the intersections of $5$-circuits in the necessary depth.

As a byproduct we obtain a stronger result for cubic graphs that are 
cyclically $4$-edge-connected.
\begin{theorem}\label{thm:main2}
Let $G$ be a cyclically $4$-edge-connected cubic graph on $m$ edges. 
Then $G$ has a cycle cover consisting of three cycles of length at most $47/30 \cdot m \approx 1.567 m$.
\end{theorem}

\medskip 

The paper is organised as follows. 
In Section~2 we show how one can modify $(\Z_2 \times \Z_2)$-flows.
We prove all flow manipulating lemmas used in the rest of the paper. 
Given a fixed $2$-factor, to perform flow manipulating operations 
we have to process some of the $5$-circuits of the $2$-factor in a specific order, 
while some others have to be processed in pairs. 
In Section~3 we specify the required order and pairs.
In Section~4 we use the lemmas from Section~2 to construct a $(\Z_2\times Z_2)$-flow 
that has special properties with respect to the $2$-factor and how the pairs and order of $5$-circuits were chosen.  
In Section~5 we define two cycle covers 
based on the $2$-factor, 
on the way how the pairs and the order of $5$-circuits were chosen, 
and on the flow. 
We bound the length of the cycle covers by using these properties.
We combine the two bounds to obtain a new bound that depends only on the number of edges in the graph
and the number of $5$-circuits in the $2$-factor that do not intersect other $5$-circuits in the graph.
Finally, in Section~6 we show that there exists a $2$-factor with few $5$-circuits that do not intersect other $5$-circuits in the graph, which completes the proofs of Theorems~\ref{thm:main} and~\ref{thm:main2}.

\section{Chain extensions and modifications}

Let $G$ be a graph and $v$ be a vertex of $G$. Let $\partial_G(v)$ be the set of edges 
with exactly one end in $v$. For a set $A \subseteq V(G)$, let $\partial_G(A)$ be 
the set of edges with exactly one end in $A$. 
For a subgraph $H$ of let $\partial_G(H)$ be the multiset of edges from $E(G)-E(H)$
that contains each edge $e \in E(G)$ as many times as is the number 
of vertices from $V(H)$ incident with $e$. 
Moreover, if $G$ is clear from context, we omit $G$ from the notation.
To work with multisets that arise from the definition of $\partial_G(H)$ we use the following operations. 
In the operation $A \cap B$ ($A-B$), the first operand $A$ may be a multiset, while the second operand $B$
is always an ordinary set. By $A \cap B$ ($A-B$) we denote the multiset that contains elements from $A$
that are in $B$ (that are not in $B$) the same number of times as they are in $A$.

All flows in this article are in $(\Z_2\times\Z_2)$. 
We use colour symbols $\cR$, $\cG$, $\cB$ for non-zero elements of $(\Z_2\times\Z_2)$ and $0$ for $(0,0)$.
As $\Z_2\times\Z_2$ has only involutions  we use the simplified definition of the notions chain and flow, omitting orientations. 
Let $G$ be a graph.
A \emph{chain} $\phi$ on $G$ is a function $\phi \ : \ E(G) \to \Z_2\times\Z_2$.
A vertex $v \in V(G)$ is a \emph{zero-sum vertex} in $\phi$ if $\sum_{e\in \partial_G(v)} \phi(e)=0$.
A \emph{flow} $\phi$ is a chain whose all vertices are zero-sum in $\phi$. 
For a chain $\phi$, let the \emph{zeroes of $\phi$}, denoted by $Z(\phi)$, be the set of all edges $e\in E(G)$
such that $\phi(e) = 0$.
A \emph{nowhere-zero} flow is a flow $\phi$ such that $Z(\phi)=\emptyset$.
Let $H$ be a connected subgraph of $G$ and let $\phi$ 
be a chain on $G$. The subgraph $H$ has a \emph{nowhere-zero boundary} 
in $\phi$ if $\partial_G(H) \cap Z(\phi)=\emptyset$.
The chain $\phi$ is \emph{$H$-extensible} if  each $v \in V(G)-V(H)$ is zero-sum in $\phi$.
We call a flow $\phi'$ on $G$ 
\emph{an $H$-extension of the $H$-extensible chain $\phi$} if for each $e\in E(G)- E(H)$ $\phi(e)=\phi'(e)$.

The Parity Lemma restricts the flow values on the edges from $\partial_G(H)$ in an $H$-extensible chain.
\begin{lemma}\label{parity}
Let $G$ be a graph, let $H$ be a connected subgraph of $G$ and let $\phi$ be a $H$-extensible chain on $G$.
Then for each $a\in\{\cR, \cG, \cB\}$
$$
|\partial_G(H)-Z(\phi)| \equiv |\partial_G(H) \cap \phi^{-1}(a)| \ \ \ \ (\bmod{\ 2}).
$$
\end{lemma}

The following lemma allows us extend a $H$-extensible chain on $G$ into a flow.
\begin{lemma}\label{extend}
Let $G$ be a graph, let $H$ be a subgraph of $G$
and let $\phi$ be an~$H$-extensible chain on $G$. Then there exists a flow
on $G$ that is an $H$-extension of $\phi$.
\end{lemma}
\begin{proof}
Proof by induction on $|E(H)|$. If $|E(H)|=0$, then $H$ is an isolated vertex, we denote it $v$. 
As $\Z_2\times\Z_2$ has only involutions,
$$
0=\sum_{w\in V(G)-\{v\}} \sum_{e\in \partial_G(w)} \phi(e)= \sum_{e\in \partial_G(v)} \phi(e).
$$
Thus $\phi$ is the flow required by this lemma.

If $|E(H)|>0$ and $H$ is not a tree, then let $e$ be an edge of a circuit of $H$.
The induction hypothesis for the subgraph $H-e$ guarantees the existence of the required flow.

If $|E(H)|>0$ and $H$ is a tree, then let $v$ be a leaf of $H$ and let 
$e$ be the edge of $H$ incident with $v$.
We set $\sum_{e'\in \partial_G(v)-\{e\}} \phi(e')$ to be the flow value of $e$ and we remove $v$ from $H$.
The assumptions of the induction hypothesis are satisfied and the conclusion guarantees the existence of the required flow.
\end{proof}

Let $G$ be a graph, let $H$ be a connected subgraph of $G$, and let $\phi$ 
be an $H$-extensible chain. An $H$-extensible chain $\phi'$ is 
an \emph{$H$-modification} of $\phi$ if $Z(\phi)-E(H)=Z(\phi')-E(H)$.
A simple way to make a $H$-modification is by using switches of Kempe chains.
The following two lemmas help us to perform such $H$-modifications
and track the changes of the flow values on $\partial_G(H)$.
\begin{lemma}\label{mod1}
Let $G$ be a graph, let $H$ be a connected subgraph of $G$ and let $\phi$ be a $H$-extensible chain on $G$.
Let $e_1$ and $e_2$ be two distinct edges from $\partial_G(H) - Z(\phi)$ and 
let $x$ and $y$ be two distinct non-zero elements of $\Z_2\times \Z_2$ such that $\{\phi(e_1),\phi(e_2)\} \subseteq \{x, y\}$.
Assume that $e_1$ and $e_2$ have a common endvertex from $V(G)-V(H)$.
Then there exists an $H$-extensible chain $\phi'$ that is an $H$-modification of $\phi$
such that 
$\phi'(e)=\phi(e)$ for $e \in \partial_G(H)-\{e_1, e_2\}$
and
$\phi'(e)=\phi(e)+x+y$ for $e \in \{e_1,e_2\}$.
\end{lemma}
\begin{proof}
Let 
$\phi'(e)=\phi(e)$ for $e \in E(G)-\{e_1, e_2\}$
and
$\phi'(e)=\phi(e)+x+y$ for $e \in \{e_1,e_2\}$.
As both $\phi(e_1)$ and $\phi(e_2)$ are in $\{x,y\}$, both 
$\phi'(e_1)$ and $\phi'(e_2)$ are non-zero, thus $\phi'$ is the $H$-modification of $\phi$. 
For each vertex $v\in V(G)-V(H)$, $v$ is a zero-sum vertex in $\phi'$, even for the common endvertex of $e_1$ and $e_2$ as we add the same value to both $\phi(e_1)$ and $\phi(e_2)$. Thus $\phi'$ is an $H$-exensible chain.
It follows that $\phi'$ is an $H$-modification of $\phi$ with all properties required by this lemma.
\end{proof}

\begin{lemma}\label{mod2}
Let $G$ be a graph, let $H$ be a connected subgraph of $G$ and let $\phi$ be an $H$-extensible chain on $G$.
Let $f$ be an edge from $\partial_G(H) - Z(\phi)$ and 
let $x$ and $y$ be two distinct non-zero elements of $\Z_2\times \Z_2$ such that $\phi(f) \in \{x, y\}$.
Then there exists an edge $f'$ with $\phi(f')\in\{x,y\}$ 
such that the multiset $\{f,f'\}$ is a submultiset of the multiset $\partial_G(H) - Z(\phi)$,
and an $H$-extensible chain $\phi'$ that is an $H$-modification of $\phi$
such that 
$\phi'(e)=\phi(e)$ for $e \in \partial_G(H)-\{f, f'\}$
and
$\phi'(e)=\phi(e)+x+y$ for $e \in \{f,f'\}$.
\end{lemma}
\begin{proof}
If $f$ has both endvertices on $V(H)$, then we set $f'=f$ and define 
$\phi'(f)=\phi(f)+x+y$ and $\phi'(e)=\phi(e)$ for edges $e$ other than $f$.
This chain fulfills the conditions of the lemma.

Thus assume that $f$ has one endvertex $v_1$ in $H$ and the other endvertex
$v_2$ outside $H$. 
Let $A$ be the subgraph of $G-V(H)$ induced by all the edges $e$ with $\phi(e)\in\{x,y\}$.
By Lemma~\ref{parity}, the component $B$ of $A$ that contains $v_2$ 
has $|\partial_G(B)\cap(\phi^{-1}(x) \cup \phi^{-1}(y))|$ even.
Therefore there is an edge $f'\in \partial_G(B)$ other than $f$ with $\phi(f')\in\{x,y\}$.
Let $w_2$ we the endvertex of $f'$ in $B$. The second endvertex must be in $H$.
Let $P$ we a path between $v_2$ and $w_2$ in $B$.

We define $\phi'$ as follows:
$\phi'(e)=\phi(e)$, for $e \in E(G)-\{f, f'\}-E(P)$,
and
$\phi'(e)=\phi(e)+x+y$, for $e \in \{f,f'\} \cup E(P)$.
As the edges $e\in\{f,f'\} \cup E(P)$ have $\phi(e)\in\{x,y\}$, also
$\phi'(e)\in\{x,y\}$ and thus $\phi'$ is non-zero on these edges. Moreover, each 
$v\in V(G)-V(H)$ is zero-sum in $\phi'$ as compared to $\phi$
the same value from $\Z_2\times\Z_2$ is added to exactly two edges from $\partial_G(v)$.
Thus we conclude that $\phi'$ is an $H$-modification of $\phi$ required by this lemma.
\end{proof}

\bigskip

Let $H$ be a connected subgraph of a graph $G$ an let $\phi$ be an $H$-extensible chain on $G$.
To facilitate the discussion on chain extensions we order the edges 
from the multiset $\partial_G(H)$ 
(edges occurring twice in the multiset will occur twice in the ordering). 
We call such ordering an \emph{$H$-boundary-ordering}.
If vertices of $H$ have degree three in $G$ and degree two or three in $H$, then we can give an $H$-boundery-ordering by a sequence of vertices of degree two in $H$---if we have such a sequence of vertices $v_0, v_1, \dots, v_{|\partial_G(H)|-1}$, then we obtain the
$H$-boundary-ordering  $e_0, e_1, \dots, e_{|\partial_G(H)|-1}$ by setting, for each $i$, 
$e_i$ to be the unique edge outside $H$ incident with $v_i$. 
Given an $H$-boundary-ordering $o$ and an $H$-extensible chain $\phi$,
we define $\phi(H,o)$ to be the string of elements $a_0a_1\dots a_{|\partial_G(H)|-1}$ 
from $\Z_2\times Z_2$ such that, for each $i$, $\phi(e_i)=a_i$. 
If $H$ is a circuit then a \emph{natural $H$-boundary-ordering} is an $H$-boundary-ordering
defined by a vertex-sequence where the vertices of $H$ are in a consecutive order in the circuit.

The following lemma allows us to extend a chain to some $5$-circuits without introducing a new zero value.
\begin{lemma}\label{lemma5c}
Let $G$ be a graph, 
let $C$ be a $5$-circuit of $G$ that contains only vertices of degree three in $G$, 
let $\phi$ be a $C$-extensible chain on $G$ such that $C$ has a nowhere-zero boundary in $\phi$,
let $o$ be a natural $C$-boundary-ordering, 
and let $\phi(C,o)=a_0a_1a_2a_3a_4$.
If $a_i=a_{i+1}=a_{i+2}$, for some $i\in\{0, \dots, 4\}$ with indices taken modulo $5$,
then there exist a flow $\phi'$ on $G$ that is a $C$-extension of $\phi$ such that
$Z(\phi') \cap E(C) = \emptyset$.
\end{lemma}
\begin{proof}
Let us denote the vertices of $C$ so that $o$ may be defined  by the vertex-sequence $v_0, v_1, \dots, v_4$.
Due to the symmetry of $C$ we may without loss of generality assume $a_0=a_1=a_2$. 
As $\Z_2\times \Z_2$ has automorphisms between non-zero elements, we can without loss of generality assume $a_0=\cR$. Then by Lemma~\ref{parity} $\{a_3,a_4\}=\{\cG,\cB\}$.
As even with element $\cR$ fixed there is an automorphism between $\cG$ and $\cB$
in $\Z_2\times \Z_2$, we may, again, without loss of generality assume $a_3=\cG$ and
$a_4=\cB$. We define the flow $\phi'$ as follows.
Let $\phi'(e)=\phi(e)$ if $e\not\in E(C)$, $\phi(v_0v_1)=\cB$, $\phi(v_1v_2)=\cG$,
$\phi(v_2v_3)=\cB$, $\phi(v_3v_4)=\cR$ and $\phi(v_4v_0)=\cG$.
\end{proof}

While Lemma~\ref{lemma5c} allows us to extend a chain into a flow on a $5$-circuit $C$
provided that the edges from $\partial(C)$ have right chain values, the following lemma
allows us to modify a chain into a flow regardless of the chain values on $\partial(C)$. 
On the other hand, we require the existence of a special vertex next to the circuit. 

\begin{lemma}\label{l5u}
Let $G$ be a bridgeless graph and let $C$ be a $5$-circuit of $G$ that contains only 
vertices of degree three in $G$. Assume that there is a vertex $v \in V(G)-V(C)$ that is neighbouring
with at least two vertices of $C$. 
Let $\phi$ be an $C$-extensible chain on $G$ such that $C$ has a nowhere-zero boundary in $\phi$. 
There exists a flow $\phi'$ on $G$ that is a $C$-modification of $\phi$
such that $Z(\phi') \cap E(C) = \emptyset$.
\end{lemma}
\begin{proof}
Let $C=v_0v_1v_2v_3v_4v_0$. Let $o=e_0e_1e_2e_3e_4$ be the $C$-boundary-ordering given by 
the vertex sequence $v_0$, $v_1$, $v_2$, $v_3$, $v_4$. 
We may without loss of generality assume that $v_0$ is a neighbour of $v$ and either $v_1$ or $v_2$ is a neighbour of $v$. 
Let $\phi(C,o)=a_0a_1a_2a_3a_4$. By Lemma~\ref{parity}, each element
from $\{\cR,\cG,\cB\}$ is used odd number of times in $\phi(C,o)$.
Due to automorphisms of $\Z_2\times\Z_2$ we may without loss of generality assume
that $\cR$ is used three times in $\phi(C,o)$ and $\cG$ is before $\cB$ in $\phi(C,o)$. 
Moreover, we may assume that
$\phi(C,o) \in \{\cR\cG\cR\cR\cB, \cG\cR\cB\cR\cR, \cR\cG\cR\cB\cR, 
\cR\cR\cG\cR\cB, \cG\cR\cR\cB\cR\}$, otherwise Lemma~\ref{lemma5c} gives
the desired $C$-modification of $\phi$. 
We conclude the proof by showing that in each case it is possible to make 
a $C$-modification of $\phi$ so that we can use Lemma~\ref{lemma5c} to get the 
$C$-modification of $\phi$ required by this lemma.

\medskip

Assume first that $v_0$ and $v_1$ are neighbors of $v$. 
Due to the symmetries of $C$, considering only automorphisms stabilizing the set $\{v_0,v_1\}$,  
it suffices to assume 
$\phi(C,o) \in \{\cR\cG\cR\cR\cB, \cR\cG\cR\cB\cR, \cR\cR\cG\cR\cB\}$.

If $\phi(C,o)=\cR\cG\cR\cR\cB$, then by Lemma~\ref{mod1}, 
with $e_0$ and $e_1$ as the two edges and $\cR$ and $\cG$ as the two elements of 
$\Z_2\times\Z_2$,
there exist a $C$-modification $\phi_1$ of $\phi$ such that $\phi_1(C,o)=\cG\cR\cR\cR\cB$. The chain $\phi_1$ satisfies the conditions of Lemma~\ref{lemma5c} as required.

If $\phi(C,o)=\cR\cG\cR\cB\cR$, then by Lemma~\ref{mod2}, 
with $e_3$ as the edge and $\cR$ and $\cB$ as the two elements of $\Z_2\times\Z_2$
there exist a $C$-modification $\phi_1$ of $\phi$ such that 
$\phi_1(C,o)\in\{\cB\cG\cR\cR\cR, \cR\cG\cB\cR\cR, \cR\cG\cR\cR\cB\}$.
If $\phi_1(C,o)\in\{\cB\cG\cR\cR\cR, \cR\cG\cB\cR\cR\}$, then $\phi_1$
is the sought modification of $\phi$, satisfying the conditions of Lemma~\ref{lemma5c}. 
If $\phi_1(C,o)=\cR\cG\cR\cR\cB$, then we proceed according to the case when
$\phi(C,o)=\cR\cG\cR\cR\cB$.

If $\phi(C,o)=\cR\cR\cG\cR\cB$, then by Lemma~\ref{mod1}, 
with $e_0$ and $e_1$ as the two edges and $\cR$ and $\cG$ as the two elements of $\Z_2\times\Z_2$,
there exist a $C$-modification $\phi_1$ of $\phi$ such that $\phi_1(C,o)=\cG\cG\cG\cR\cB$ which is the sought modification of $\phi$, satisfying the conditions of Lemma~\ref{lemma5c}.

\medskip

On the other hand, assume $v_0$ and $v_2$ are neighbors of $v$. 
Due to symmetries of $C$, considering only automorphisms stabilizing the set $\{v_0,v_2\}$,  
it suffices to assume 
$\phi(C,o) \in \{\cR\cG\cR\cR\cB, \cG\cR\cB\cR\cR,  
\cR\cR\cG\cR\cB\}$.

If $\phi(C,o)=\cR\cG\cR\cR\cB$, then by Lemma~\ref{mod1}, 
with $e_0$ and $e_2$ as the two edges and $\cR$ and $\cG$ as the two elements of $\Z_2\times\Z_2$,
there exist a $C$-modification $\phi_1$ of $\phi$ such that $\phi_1(C,o)=\cG\cG\cG\cR\cB$, which is the sought modification of $\phi$, satisfying the conditions of Lemma~\ref{lemma5c}.

If $\phi(C,o)=\cG\cR\cB\cR\cR$, then by Lemma~\ref{mod2}, 
with $e_1$ as the edge and $\cR$ and $\cG$ as the two elements of $\Z_2\times\Z_2$
there exist a $C$-modification $\phi_1$ of $\phi$ such that 
$\phi_1(C,o)\in\{\cR\cG\cB\cR\cR, \cG\cG\cB\cG\cR, \cG\cG\cB\cR\cG\}$.
If $\phi_1(C,o)\in\{\cR\cG\cB\cR\cR, \cG\cG\cB\cR\cG\}$, then $\phi_1$
is the sought modification of $\phi$ satisfying the conditions of Lemma~\ref{lemma5c}. 
If $\phi_1(C,o)=\cG\cG\cB\cG\cR$, then we proceed according to the case when
$\phi(C,o)=\cR\cR\cG\cR\cB$.

If $\phi(C,o)=\cR\cR\cG\cR\cB$, then by Lemma~\ref{mod1}, 
with $e_0$ and $e_2$ as the two edges and $\cR$ and $\cG$ as the two elements of $\Z_2\times\Z_2$,
there exists a $C$-modification $\phi_1$ of $\phi$ such that $\phi_1(C,o)=\cG\cR\cR\cR\cB$ which is the sought modification of $\phi$, satisfying the conditions of Lemma~\ref{lemma5c}.
\end{proof}

\bigskip

The argumentation where we describe how the desired $C$-modification is constructed using Lemmas~\ref{mod1}~and~\ref{mod2}
as written in the proof of Lemma~\ref{l5u} is very repetitive. 
To avoid repetition, we introduce a way how to certify that a modification of a chain satisfying some property can always be made.

Let $G$ be a graph, $H$ be a  subgraph of $G$,
$o=e_0\dots e_{|\partial_G(H)|-1}$ be an $H$-boundary-ordering and
let $\phi$ be an $H$-extensible chain.
If Lemma~\ref{mod1} is used with $e_i$ and $e_j$ as the two edges,
with $x$ and $y$ as the two elements of $\Z_2\times\Z_2$,
and Lemma~\ref{mod1} produces a chain $\phi'$,
then we write 
\begin{center}
\begin{tabular}{ccc}
$\phi(H,o)$& \myto{$xy$} & $\phi'(H,o)$
\end{tabular}
\end{center}
with $i$-th and $j$-th element of $\phi(H,o)$ underlined.
If Lemma~\ref{mod2} is used
with $e_i$ as the edge, with $x$ and $y$ as the two elements of $\Z_2\times\Z_2$,
and Lemma~\ref{mod2} produces a chain $\phi'$,
then we do not know exactly what $\phi'(H,o)$ is, but 
Lemma~\ref{mod2} gives us a list of possibilities
$A_0$, $A_1$, $\dots$ such that for some $j$,
$\phi'(H,o)=A_j$.
We record this as follows \newline
\begin{center}
\begin{tabular}{ccc}
$\phi(H,o)$ &\myto{$xy$}& $A_0$\\
            &           & $A_1$\\
            &           & $\dots$,              
\end{tabular}
\end{center}
with $i$-th element of $\phi(H,o)$ underlined.
Note that we know which lemma was used according to the number of underlined edge colours.
We finish each line by indicating the reason why the resulting 
chain modification has the desired property. If this last entry in a line is a sequence $A$ in brackets,
then it indicates that we complete the modification as in case $A$ (of course, we have to avoid cyclic references). It is possible that
$A$ differs from the value before $A$: an automorphism of $\Z_2\times\Z_2$ or/and 
an automorphism of $H$ might be applied,
provided that the autoromphisms preserve the desired property 
and preserve when Lemma~\ref{mod1} can be applied 
(e.g. in the proof of Lemma~\ref{l5u} we could use only automorphisms stabilising the set of vertices incident with
the two edges from $\partial(C)$ with the common neighbour). 

\medskip

As an example, we present how the modifications of $\phi$ from 
the proof of Lemma~\ref{l5u} in case $v_0$ and $v_1$
have a common neighbour outside $C$ are recorded. 
\begin{center}
\begin{tabular}{cccc}
\underline{RG}RRB &\myto{RG}& GRRRB & Lemma~\ref{lemma5c}\\
RGR\underline{B}R &\myto{RB}& BGRRR & Lemma~\ref{lemma5c}\\
                  &         & RGBRR & Lemma~\ref{lemma5c}\\
                  &         & RGRRB & [RGRRB]\\
\underline{RR}GRB &\myto{RG}& GGGRB & Lemma~\ref{lemma5c}
\end{tabular}
\end{center}
In case $v_0$ and $v_2$ have a common neighbour outside $C$ the modifications from Lemma~\ref{l5u} 
are recorded as follows. 
\begin{center}
\begin{tabular}{cccc}
\underline{R}G\underline{R}RB & \myto{RG} & GGGRB & Lemma~\ref{lemma5c}\\
G\underline{R}BRR             & \myto{RG} & RGBRR & Lemma~\ref{lemma5c}\\
                              & \myto{RG} & GGBGR & [RRGRB]\\
                              & \myto{RG} & GGBRG & Lemma~\ref{lemma5c}\\
\underline{R}R\underline{G}RB & \myto{RG} & GRRRB & Lemma~\ref{lemma5c}\\
\end{tabular}
\end{center}
\bigskip

We prove two lemmas of similar flavor, one concerning $6$-circuits and the second one concerning a special subgraph containing three $5$-circuits.
\begin{lemma}\label{l6}
Let $G$ be bridgeless graph,
let $C$ be a $6$-circuit of $G$ that contains only vertices of degree three in $G$ and
let $\phi$ be a $C$-extensible chain on $G$ such that $C$ has nowhere-zero boundary in $\phi$.
Let $o$ be a natural $C$-boundary ordering.
There exists a $C$-extensible chain $\phi'$ on $G$ that is a $C$-modification of $\phi$
such that $\phi(C,o)=a_0a_1a_2a_3a_4a_5$ satisfies at least one of the conditions from Figure~\ref{fig6}.
\end{lemma}
\begin{figure}
\begin{center}
\begin{tabular}{|c|c|}
\hline
No.&Condition\\
\hline
1&$a_0=a_1$, $a_2=a_3$, $a_4=a_5$\\
2&$a_0=a_5$, $a_1=a_2$, $a_3=a_4$\\
\hline
3&$a_0=a_3$, $a_1=a_2$, $a_4=a_5$\\
4&$a_0=a_5$, $a_1=a_4$, $a_2=a_3$\\
5&$a_0=a_1$, $a_2=a_5$, $a_3=a_4$\\
\hline
6&$a_0=a_3$, $a_1=a_4$, $a_2=a_5$\\
\hline
\end{tabular}
\caption{The conditions that $\phi(C,o)=a_0a_1a_2a_3a_4a_5$ has to satisfy in Lemma~\ref{l6}.}\label{fig6}
\end{center}
\end{figure}
\begin{proof}
We may without loss of generality assume that 
$|\partial(C) \cap \phi^{-1}(\cR)| \ge |\partial(C) \cap \phi^{-1}(\cG)| \ge |\partial(C) \cap \phi^{-1}(\cB)|$.
By Lemma~\ref{parity}, 
$(|\partial(C) \cap \phi^{-1}(\cR)|, |\partial(C) \cap \phi^{-1}(\cG)|, |\partial(C) \cap \phi^{-1}(\cB)|)
\in \{(6,0,0), (4,2,0), (2,2,2)\}$.
Note that the set of conditions in Figure~\ref{fig6} is preserved by automorphisms of $C$ (the lines separating the conditions on Figure~\ref{fig6} separate conditions in different orbits under action of the automorphism group of $C$). 
Thus among the symmetric cases arising from different consecutive vertex orders
we consider only the cases where $\phi(C,o)$ is alphabetically maximal.
In what follows we show that in all these cases we can satisfy one of the conditions in 
Figure~\ref{fig6}. The values of $\phi(C,o)$ are sorted by $|\partial(C) \cap \phi^{-1}(\cR)|$
and then by the reverse alphabetical order. Note that $a_0=\cR$ and if $a_1\neq \cR$ then 
no colour can occour at two consecutive positions $\phi(C,o)$.
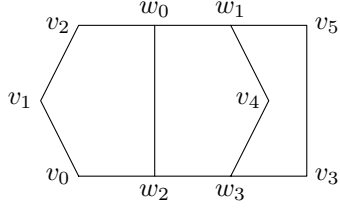
\begin{center}
\begin{tabular}{cccc}
RRRRRR &&& Condition 1\\
RRRRGG &&& Condition 1\\
RRRG\underline{R}G &\myto{RB}& BRRGBG &[RRGBGB]\\
                   &         & RBRGBG &[RGRBGB]\\
                   &         & RRBGBG &[RRGBGB]\\
RRGRRG &&& Condition 3\\                   
RRGGBB &&& Condition 1\\
RR\underline{G}BGB &\myto{GB}& RRBGGB & Condition 5\\
                   &         & RRBBBB & Condition 1\\
                   &         & RRBBGG & Condition 1\\
RRGBBG &&& Condition 5\\
\underline{R}GRBGB &\myto{RB}& BGBBGB & Condition 4\\
                   &         & BGRRGB & Condition 4\\
                   &         & BGRBGR & Condition 6\\
RGBRGB &&& Condition 6\\
\end{tabular}
\end{center}
\end{proof}

If $H$ is a subgraph of $G$ isomorphic to graph $S$ from Figure~\ref{fig10}, then
a natural $H$-boundary-ordering is an ordering of $\partial_G(H)$ given by the
vertex-sequence $v'_0$, $v'_1$, $v'_2$, $v'_3$, $v'_4$, $v'_5$, where vertices of $H$ 
are the isomorphic images the vertices $v_0$, $v_1$, $v_2$, $v_3$, $v_4$, $v_5$  from $S$, respectively.

\begin{lemma}\label{l10}
Let $G$ be bridgeless graph,
let $H$ be a subgraph of $G$ isomorphic to graph $S$ from Figure~\ref{fig10} that contains only vertices of degree three in $G$ and
let $\phi$ be an $H$-extensible chain on $G$ such that $H$ has a nowhere-zero boundary in $\phi$.
Let $o$ be a natural $H$-boundary-ordering.
There exists an $H$-extensible chain $\phi'$ on $G$ that is as an $H$-modification of $\phi$
such that $\phi(H,o)=a_0a_1a_2a_3a_4a_5$ satisfies at least one of the conditions from Figure~\ref{fig10}.
\end{lemma}

\begin{figure}
\begin{center}
\includegraphics{fig-1}

\bigskip

\begin{tabular}{|c|c|}
\hline
No.&Condition\\
\hline
1&$a_0=a_1, a_2=a_3, a_4=a_5$\\
2&$a_0=a_5, a_1=a_2, a_3=a_4$\\
\hline
3&$a_0=a_1, a_2=a_4, a_3=a_5$\\
4&$a_0=a_4, a_1=a_2, a_3=a_5$\\
\hline
5&$a_0=a_3, a_1=a_5, a_2=a_4$\\
6&$a_0=a_4, a_1=a_3, a_2=a_5$\\
\hline
7&$a_0=a_4, a_1=a_5, a_2=a_3$\\
8&$a_0=a_5, a_1=a_3, a_2=a_4$\\
\hline
\end{tabular}
\caption{Graph $S$ and the conditions that $\phi(H,o)=a_0a_1a_2a_3a_4a_5$ has to satisfy in Lemma~\ref{l10}.}\label{fig10}
\end{center}
\end{figure}

\begin{proof} 
We may without loss of generality assume that 
$|\partial(H) \cap \phi^{-1}(\cR)| \ge |\partial(H) \cap \phi^{-1}(\cG)| \ge |\partial(H) \cap \phi^{-1}(\cB)|$.
By Lemma~\ref{parity}, 
$(|\partial(H)\cap\phi^{-1}(\cR)|, |\partial(H) \cap \phi^{-1}(\cG)|, |\partial(H) \cap \phi^{-1}(\cB)|)
\in \{(6,0,0), (4,2,0), (2,2,2)\}$.
Note that the set of conditions in Figure~\ref{fig10} is preserved by automorphisms of $H$
(besides identity, the only other automorphism takes $v_0$ to $v_2$ and $v_3$ to $v_5$; the lines separating the conditions on Figure~\ref{fig6} separate conditions in different orbits under action of the automorphism group of $H$). 
Among the symmetric cases arising from different vertex orders
we consider only the cases where $\phi(H,o)$ is alphabetically maximal.
In what follows we show that in all these cases we can satisfy one of the conditions from 
Figure~\ref{fig10}. The values of $\phi(H,o)$ are sorted by $|\partial(H) \cap \phi^{-1}(\cR)|$
and then by the reverse alphabetical order. Note that if $a_0a_1a_2=a_2a_1a_0$, then we can require
that $a_3a_4a_5$ is alphabetically behind $a_5a_4a_3$.
Also note that the case when
$\phi(H,o)=\text{RGBGRB}$ is symmetric to the case when $\phi(H,o)=\text{RGBRBG}$
and the case when $\phi(H,o)=\text{RGBBRG}$ is symmetric to the case when $\phi(H,o)=\text{RGBGBR}$.

\begin{center}
\begin{longtable}{cccc}
RRRRRR &&& Condition 1\\
RRRRGG &&& Condition 1\\
RRRGRG &&& Condition 3\\
RRGRRG &&& Condition 6\\
RRGRGR &&& Condition 5\\
RRGGRR &&& Condition 1\\
RGRRRG &&& Condition 5\\
RGRR\underline{G}R &\myto{RG}& GGRRRR & Condition 1\\
                   &\myto{RG}& RRRRRR & Condition 1\\
                   &\myto{RG}& RGGRRR & Condition 2\\
                   &\myto{RG}& RGRGRR & Condition 6\\
                   &\myto{RG}& RGRRRG & Condition 5\\
RGGRRR &&& Condition 1\\
G\underline{R}GRRR &\myto{GR}& RGGRRR & Condition 4\\
                   &\myto{GR}& GGRRRR & Condition 1\\
                   &\myto{GR}& GGGGRR & Condition 1\\
                   &\myto{GR}& GGGRGR & Condition 3\\
                   &\myto{GR}& GGGRRG & Condition 2\\
RRGGBB &&&  Condition 1\\
RRGBGB &&&  Condition 3\\
RR\underline{G}BBG &\myto{GB}& RRBGBG & Condition 3\\
                   &\myto{GB}& RRBBGG & Condition 1\\
                   &\myto{GB}& RRBBBB & Condition 1\\
\underline{R}GRGBB &\myto{GR}& GRRGBB & [RRGBBG]\\
                   &\myto{GR}& GGRRBB & Condition 1\\
                   &\myto{GR}& GGGGBB & Condition 1\\
\underline{R}GRBGB &\myto{GR}& GRRBGB & Condition 4\\
                   &\myto{GR}& GGRBRB & Condition 3\\
                   &\myto{GR}& GGGRGR & Condition 3\\
RGBR\underline{G}B &\myto{GR}& GGBRRB & [RRGBBG]\\
                   &\myto{GR}& RRBRRB & Condition 6\\
                   &\myto{GR}& RGBGRB & Condition 6\\
RGBRBG &&& Condition 5\\
RGBGBR &&& Condition 8\\
R\underline{G}BBGR &\myto{RG}& GRBBGR & Condition 7\\
                   &\myto{RG}& RRBBRR & Condition 1\\
                   &\myto{RG}& RRBBGG & Condition 1\\
\end{longtable}
\end{center}
\end{proof}

We conclude this section by the following reformulation of a results of Fan \cite{fan}.
\begin{lemma}\label{lfan}
Let $G$ be a bridgeless graph and let $C$ be a circuit of $G$ that contains only 
vertices of degree three in $G$.  
Let $\phi$ be an $C$-extensible chain on $G$. 
\begin{enumerate}
\item There exists a flow $\phi'$ on $G$ that is a $C$-extension of $\phi$
such that $|Z(\phi') \cap E(C)| \le |E(C)|/4$ and
\item If $|C|<20$, then 
there exists a flow $\phi''$ on $G$ that is a $C$-modification of $\phi$ such that 
$|Z(\phi'') \cap E(C)| < |E(C)|/4$.
\end{enumerate}
\end{lemma}
\begin{proof}
By Lemma~\ref{extend}, $\phi$ can be extended into a flow $\phi_1$ on $G$.
By pigeonhole principle there is an element $x \in \{0,\cR,\cG,\cB\}$ such that
$|\phi^{-1}(x) \cap E(C)| \le |E(C)|/4$. We set $\phi'(e)=\phi_1(e)$ for $e\not\in E(C)$
and $\phi'(e)=\phi_1(e)+x$ otherwise. The resulting flow satisfies the first part of the lemma.

Lemma 3.1 and Theorem 3.3 of \cite{fan} guarantee the existence of
a flow $\phi''$ which is a $C$-modification of $\phi_1$ (which is also a $C$-modification of $\phi$)
satisfying the second statement of  this lemma\footnote{Lemma 3.1 and Theorem 3.3 in \cite{fan} are stated for $4$-flows. However, well known results of Tutte (see e.g. 
Corollary 6.3.2 and Theorem 6.3.3. of \cite{diestel}) guarantee that we can convert a $(\Z_2\times\Z_2)$-flow into a $4$-flow and
vice versa without changing which edges have zero values assigned.}.
\end{proof}

\section{Intersecting $5$-circuits}

Let $G$ be a bridgeless cubic graph. Let $\C^I_G$ be the set of $5$-circuits of $G$ that intersect some other $5$-circuit of $G$ in exactly one edge or in exactly two adjacent edges ($I$ stands for ``intersecting''). 
Let $\C^N_G$ be the set of all $5$-circuits on $G$ that do not intersect other $5$-circuit from $G$
($N$ stands for ``non-intersecting'').
In case $G$ is clear from the context, we sometimes omit $G$ in $\C^I_G$ and $\C^N_G$ and in several derived notations used further in the paper.

Although, not all $5$-circuits are in $\C^I$ or $\C^N$, the circuits in the $2$-factors of our interests are 
either from $\C^I$ or from $\C^N$. 
By $G/F$ we denote the graph created by \emph{contracting $F$ in $G$}, 
that is by contracting all the edges from $E(F)$ in $G$. A graph is $5$-odd-edge-connected if for each
$A\subseteq V(G)$ $|\partial_G(A)| \not\in \{1, 3\}$.

\begin{lemma}
Let $G$ be a bridgeless cubic graph and let $F$ be a $2$-factor of $G$ such that $G/F$ is $5$-odd-edge-connected.
All $5$-circuits in $F$ are either from $\C^I_G$ or $\C^N_G$.
\end{lemma}
\begin{proof}
Let $C$ be a $5$-circuit of $F$ that is not in $\C^N$. Thus $C$ intersects a $5$-circuit $D$ of $G$.
If $|E(C) \cap E(D)|=4$, then $|\partial(V(C))| = 3$ (or $1$ if $C$ has two chords). This contradicts the $5$-odd-edge-connectivity of $G/F$.
If $E(C) \cap E(D)$ are three edges consecutive on $C$, then there is exactly one vertex in $V(D)-V(C)$.
This vertex has two neighbours in $C$, which contradicts with $C$ being a $2$-factor.
If $E(C) \cap E(D)$ are three non-consecutive edges, then $|\partial(V(C))| = 1$, a contradiction with $G$ being bridgeless.
If $E(C) \cap E(D)$ are two non-consecutive edges, then the unique vertex in $V(D)-V(C)$ has two neighbours in $C$,  which contradicts with that $C$ is in a $2$-factor.
In all other cases $C$ is in $\C^I$ as required by the lemma statement.
\end{proof}

Let $F$ be a $2$-factor of $G$.
Let us denote $\C^N_G(F)$ the set of circuits from $\C^N_G$ that are in $F$ 
and $\C^I_G(F)$ the set of circuits from $\C^I_G$ that are in $F$.
Let $C\in\C^I(F)$. Consider the set ${\cal S}(C)$ of all circuits $C^*$ of $F$ other than $C$ such that there exists a $5$-circuit in $G$ that intersects both $C$ and $C^*$. 
Due to the definition of $\C^I(F)$, ${\cal S}(C)\neq \emptyset$.

In Section~4, we will create a flow on $G$ by contracting circuits of $F$ and then we undo these contractions one by one. When we restore a circuit from $\C^I_G(F)$, we would like to use Lemma~\ref{l5u}. 
The condition of Lemma~\ref{l5u} requiring two boundary edges towards sharing a vertex is satisfied when
some circuit in ${\cal S}(C)$ was not restored before restoring $C$. Thus it is critical to define order in which circuits from $\C^I_G(F)$ are restored. Unfortunately, it is not possible to define an ordering that would allow us
to use Lemma~\ref{l5u} in every case. However, for the remaining cases it turns out that we can form disjouint pairs of such $5$-circuits so that Lemma~\ref{l10} is applicable.

We fix a complete strict order $>_1$ on $\C^I_G(F)$.
For $C\in \C^I_G(F)$ we define $f_{>_1}(C)$ as follows. If ${\cal S}(C) - \C^I(F) \neq \emptyset$, then we set $f_{>_1}(C)$ to be an arbitrary circuit from ${\cal S} - \C^I(F)$.
Otherwise we set $f_{>_1}(C)$ to be the smallest element of ${\cal S}$ in $>_1$.
Let $\C^P_G(F,>_1)$ be the set of circuits $C\in \C^I_G(F)$ such that $f_{>_1}(C)\in C^P_G(F)$ 
and $f_{>_1}(f_{>_1}(C))=C$ ($P$ stands for ``paired''). As on $\C^P_G(F,>_1)$ the function $f_{>_1}$ is an involution without fixed point thus it naturally partitions the circuits of $\C^P_G(F,>_1)$ into pairs. 
Let ${\cal P}_G(F, >_1)=\{\{C, f_{>_1}(C)\}; \ C \in \C^P_G(F, >_1)\}$. 
Let $\C^U_G(F,>_1)=\C^I_G(F)-\C^P_G(F,>_1)$ 
($U$ stands for ``unpaired'').
We call circuits that are in the same set from ${\cal P}_G(F, >_1)$ \emph{paired}.

Newt we show that paired circuits are in some graph isomorphic to the graph $S$ from Figure~\ref{fig10},
so that Lemma~\ref{l10} is applicable. 

\begin{lemma}\label{cl3p}
Let $G$ be a bridgeless cubic graph and let $F$ be a $2$-factor of $G$ 
such that $G/F$ is $5$-odd-edge-connected. Let $>_1$ be an ordering of the circuits in $\C^I_G(F)$.
Then ${\cal P}_G(F, >_1)$ is a partition of the set $\C^P_G(F, >_1)$ into two element subsets such that
for each two element subset the two circuits in this subset are contained in a subgraph isomorphic to the graph $S$ from Figure~\ref{fig10}.
\end{lemma}
\begin{proof}
Due to the definition of $\C^P(F, >_1)$, $f_{>_1}(C)$ is in $\C^P(F, >_1)$ thus the sets of ${\cal P}$
contain only elements from $\C^P(F, >_1)$. By the definition of $f_{>_1}(C)$, $f_{>_1}(C) \neq C$
and thus ${\cal P}(F, >_1)$ contains two element sets.
The sets in ${\cal P}$ are obviously non-empty and contain all elements of $\C^P(F, >_1)$. As $f_{>_1}(f_{>_1}(C))=C$ the sets in ${\cal P}$ do not intersect. Thus 
${\cal P}$ is a partition of $\C^P(F, >_1)$.

Let $\{C, C^*\} \in {\cal P}$. Due to the definition of $f_{>_1}(C)$, there exists a $5$-circuit $D$ of $G$ that intersects both $C$ and $C^*$. 
As in a cubic graph circuits may not intersect in only one vertex, $D$ intersects both $C$ and $C^*$ in two or three vertices and these vertices must be consecutive. Thus $|V(D)-V(C)-V(C^*)| \le 1$. Assume for contradiction that there exists 
a vertex $v\in V(D)-V(C)-V(C^*)$. But then two edges incident with $v$ are in $\partial(V(C) \cup V(C^*))$,
which is impossible since $C$ and $C^*$ are $2$-factor circuits.
Thus $V(D)\subseteq V(C) \cup V(C^*)$. It follows that the subgraph of $G$ with vertex set $V(C) \cup V(C^*)$
and edge set $E(C) \cup E(C^*) \cup E(D)$ is isomorphic to the graph $S$ from Figure~\ref{fig10}.
\end{proof}
It is possible, that the paired circuits are in more than one subgraph isomorphic to the graph $S$ from Figure~\ref{fig10}. We need to fix this subgraph. A function that assigns such a subgraph to two paired circuits is a \emph{$S$-assigning function}.

We conclude this sections with the following lemma that shows that for circuits from $\C^I_G(F)$ that are not paired, we can find an ordering that allows us to restore the circuits from $\C^I_G(F)$ so that Lemma~\ref{l5u} is applicable.
\begin{lemma}\label{cl3u}
Let $G$ be a bridgeless cubic graph and let $F$ be a $2$-factor of $G$ 
such that $G/F$ is $5$-odd-edge-connected. Let $>_1$ be an ordering of the circuits in $\C^I_G(F)$.
Then there exists a strict total order $>_2$ on $\C^U_G(F,>_1)$ such that
for each $C\in \C^U_G(F,>_1)$ there exists a circuit $C^*$ of $F$ such that $\partial_G(C)$ and $\partial_G(C^*)$
have at least two edges in common and  such that either
\begin{itemize}
\item $C^* \in \C^U_G(F,>_1)$ and $C >_2 C^*$, or
\item $C^* \not \in \C^U_G(F,>_1)$
\end{itemize}
\end{lemma}

\begin{proof}
Let $C\in \C^I(F)$. For a function $f$ by $f^i$ we mean the function $f$ iterated $i$-times.
Assume that for each positive integer $i$, $f_{>_1}^i(C) \in \C^I(F)$, $f_{>_1}^2(C) \neq C$, but
there exists an integer $k>2$ such that $f_{>_1}^k(C)=C$. 
By the definition of $f_{>_1}$ for circuit $f_{>_1}^i(C)$, for $i \in \{1, 3, \dots , 2k-3, 2k-1\}$, 
we have that $f_{>_1}^{i-1}(C)>_1 f_{>_1}^{i+1}(C)$ or $f_{>_1}^{i-1}(C) = f_{>_1}^{i+1}(C)$.
Assume for contradictuin that $f_{>_1}^{i-1}(C) = f_{>_1}^{i+1}(C)$. 
Then for all $j\ge i$ $f_{>_1}^{j-1}(C) = f_{>_1}^{j+1}(C)$. But thus 
$C=f_{>_1}^{2k}(C) = f_{>_1}^{2k+2}(C) = f_{>_1}^2(C)$, a contradiction. 
Thus  for $i \in \{1, 3, \dots , 2k-3, 2k-1\}$,  $f_{>_1}^{i-1}(C)>_1 f_{>_1}^{i+1}(C)$ and
$$
C=f_{>_1}^0(C)>_1 f_{>_1}^2(C)>_1 \dots >_1 f_{>_1}^{2k-2}(C)>_1 f_{>_1}^{2k}(C)=C,
$$ 
a contradiction. 

As $f_{>_1}(C) \neq C$ and as the circuits from $\C^I(F)$ such that $f_{>_1}^2(C) = C$ are in $\C^P(F,>_1)$,
we can define a partial strict order $>'_2$ on $\C^U(F,>_1)$ as follows: $C_1 >'_2 C_2$ 
if for some $i>1$, $f_{>_1}^i(C_1)=C_2$.  We arbitrarily extend
the strict partial order $>'_2$ into a strict total order $>_2$. 

Let $C\in \C^U(F,>_1)$. We show that $C^*=f_{>_1}(C)$ satisfies the conditions of this lemma.
Due to the definition of $f_{>_1}$, $C$ and $C^*$ are two circuits of $F$  
such that $\partial_G(C)$ and $\partial_G(C^*)$ have at least two edges in common.
If $C^* \not \in \C^U(F,>_1)$, then the conditions of this lemma are satisfied.
If $C^* \in \C^U(F,>_1)$, then $C>'_2 C^*$ and thus $C>_2 C^*$ as required by this lemma.
\end{proof}

\section{Finding the starting flow}

In the following sections we will describe several constructions of subgraphs, chains, and vectors with desired property. In principle, the constructions are procedural and thus we will describe them in a procedural way
using variables. Before each procedure we clarify, what are the variables in the procedure.
By $A:=B$, where $A$ is a variable we mean that the value of $A$  is set to $B$. To ease the analysis of the procedure we complement the steps of the procedure by comments which describe the actual properties of the variables at respective step. To ease the analysis of the loops in the procedure we provide properties that variables satisfy at the start of each loop.

Let $G$ be a graph, let $H$ be a subgraph of $G$. 
In the graph $G/H$ each component $X$ of $H$ is contracted into a vertex $v_X$. 
If a variable $A$ is a graph created by a contraction we assume that $A$ contains the information 
on how the contracted graph was created. Thus if a variable $B$ is graph created by contracting a subgraph
$H$ in $A$, then for a component $X$ of $H$ we may \emph{uncontract $X$ in $B$}, 
that is to set $B:=B/(E(H)-E(X))$. We denote this operation $B:=B \text{ uncontract } X$.
For variables that contain chains it will be convenient te take a relaxed stance to the range of the chain.
If $A$ is a variable that is a chain on $G$, and $H$ is a subgraph of $G$, 
then we consider $A$ to be also a chain on $H$ and on $G/H$.

Let $G$ be a bridgeless graph and let $F$ be a $2$-factor of $G$ such that $G/F$ is $5$-odd-edge-connected. 
Let $>_1$ be a complete ordering of the circuits in $\C^I_G(F)$ and let 
$s$ be an $S$-assigning function.
A flow $\phi$ is \emph{good with respect to $F$, $>_1$, and $s$} if it satisfies the following five properties:

\medskip

\noindent \emph{Property 1}: $\phi$ has zeroes only on edges of $F$.

\medskip

\noindent \emph{Property 2}: For each circuit $C$ of $F$, $|Z(\phi)\cap E(C)|\le |E(C)|/4$ and 
if $C$ has length $4$, $8$, $12$, or $16$, then the strict inequality holds.

\medskip

\noindent \emph{Property 3}: For each circuit $C\in \C^U_G(F,>_1)$ $Z(\phi) \cap E(C) = \emptyset$.

\medskip

\noindent \emph{Property 4}: For each $6$-circuit $C$ of $F$ 
\begin{itemize}
\item[\emph{4a}:] $\partial_G(C) \cap \phi^{-1}(x) = \emptyset$,  for some $x\in\{\cR, \cG, \cB\}$ or
\item[\emph{4b}:] $|E(C) \cap \phi^{-1}(x)| \equiv |E(C) \cap Z(\phi)| \ \  (\bmod \ 2)$, for all $x\in\{\cR, \cG, \cB\}$.
\end{itemize}

\medskip

\noindent \emph{Property 5}: For each subgraph $T=s(\{C, C^*\})$, where $\{C,C^*\}\in{\cal P}_G(F,>_1)$, $|E(T) \cap Z(\phi)| \le 1$ and
\begin{itemize}
\item[\emph{5a}:] $\partial_G(T) \cap \phi^{-1}(x) = \emptyset$,  for some $x\in\{\cR, \cG, \cB\}$ or
\item[\emph{5b}:] $E(T) \cap Z(\phi) = \emptyset$.
\end{itemize}

\medskip

\noindent If $F$, $>_1$, and $s$ are clear from the context, then we simply say that $\phi$ is good.

\begin{lemma}\label{lflow}
Let $G$ be a bridgeless cubic graph and let $F$ be a $2$-factor of $G$ such that $G/F$ is $5$-odd-edge-connected. 
Let $>_1$ be an ordering of the circuits in $\C^I_G(F)$ and let $s$ be an $S$-assigning function.
Then there exists a flow $\phi$ on $G$ that is good with respect to $F$, $>_1$ and $s$. 
\end{lemma}
\begin{proof}
Let $>_2$ be the total strict order whose existence is guaranteed by Lemma~\ref{cl3u}.
We construct $\phi$ using the following procedure that uses variable $H$ that is a graph and
variable $\phi'$ that is a chain. If we say that a chain satisfies Property~1, 2, 3, 4, or 5 we mean the respective property from the definition of good flow on $G$. To say this we do not require $\phi'$ 
to be a flow, nor to be defined on whole $G$. It suffices that the chain is defined on the part of $G$
where the property applies.

\begin{enumerate}
\item We set $H:=G/F$, and $\phi'$ to an arbitrary nowhere-zero flow on $H$ (such a flow exists \cite[Proposition 10]{jaeger}). Now $\phi'$ satisfies Property~1 and this remains true till the end of the procedure.

\item For each circuit $C \in \C^U_G(F,>_1)$ and in descending order with $>_2$ we perform the following subprocedure. At the start of each subprocedure $\phi'$ is a nowhere-zero flow on $H$. 
\begin{enumerate}
\item $H:=H \text{ uncontract } C$, $\phi'(e):=0$ for each $e \in E(C)$. Now $\phi'$ is a $C$-extensible chain
on $H$.
\item By Lemma~\ref{cl3u}, we may use Lemma~\ref{l5u} and by Lemma~\ref{l5u}
we can make a $C$-modification of $\phi'$ and obtain a nowhere-zero flow on $H$. We set $\phi'$ to be this flow.
\end{enumerate}

Now $\phi'$ satisfies Property~3 and this remains true until the end of the procedure.

\item For each circuit $C$ that is not of length $6$ and is not in ${\cal C}^P_G(F, >_1)$ we perform the following subprocedure. 
At the start of each subprocedure $\phi'$ is a flow on $H$.
\begin{enumerate}
\item $H:=H \text{ uncontract } C$, $\phi'(e):=0$ for each $e \in E(C)$. Now $\phi'$ is a $C$-extensible chain
on $H$.
\item By Lemma~\ref{lfan}, we can make a $C$-modification of $\phi'$ and obtain a flow on $H$ with small (i.e.~as required by Property~2) number of zero values in $E(C)$. We set $\phi'$ to be this flow.
\end{enumerate}

Now $\phi'$ satisfies Property~2 and this remains true until the end of the procedure.

\item For each $6$-circuit $C$ of $F$ we perform the following subprocedure.
At the start of each subprocedure $\phi'$ is a flow on $H$. 
\begin{enumerate}
\item $H:=H \text{ uncontract } C$, $\phi'(e):=0$ for each $e \in E(C)$. Now $\phi'$ is a $C$-extensible chain
on $H$.
\item Let us denote the vertices of $C$ by $v^C_0$, $v^C_1$, \dots, $v^C_5$ in a consecutive order and consider 
the $H$-boundary-ordering given by the vertex sequence $v^C_0$, $v^C_1$, \dots, $v^C_5$.
By Lemma~\ref{l6}, there exists a $C$-modification of $\phi'$ such that one of 
the conditions from Figure~\ref{fig6} is satisfied. We set $\phi'$ to be this $C$-modification.

\item We delete the edges of $C$ and add three edges according to the satisfied condition (if more than one condition is satisfied, we choose one of the conditions arbitrarily): 
for all three equalities $a_i=a_j$, where $i<j$, the satisfied condition requires,
we add an edge $e^C_i$ between $v^C_i$ and $v^C_j$ to $H$ and set $\phi'(e^C_i):=a_i$.
After we do this for all three equalities, $\phi'$ is a flow on $H$.
Thanks to this substitution the satisfied condition from Figure~\ref{fig6} remains satisfied till the end of the procedure for $C$.
\end{enumerate}

\item For each $\{C, C^*\}\in {\cal P}_G(F, >_1)$ we perform the following subprocedure.
Let $T=s(\{C, C^*\})$.
At the start of each subprocedure $\phi'$ is a flow on $H$. 
\begin{enumerate}
\item $H:=(H \text{ uncontract } C)$, $H:=(H \text{ uncontract } C^*)$, and
$\phi'(e):=0$ for each $e \in E(C) \cup E(C^*)$. 
Now $\phi'$ is a $T$-extensible chain on $H$.

\item Let us denote the vertices of $T$ by $v^T_i$, for $i\in \{0, \dots, 5\}$, 
and $w^T_i$, for $i\in \{0, \dots, 3\}$, so that there is an isomorphism between $S$ (from Figure~\ref{fig10}) and 
$T$ such that $v^T_i$ is the isomorphic image of $v_i$, for $i\in \{0, \dots, 5\}$, 
and $w^t_i$ is the isomorphic image of $w_i$, for $i\in \{0, \dots, 3\}$.
Consider the $H$-boundary-ordering given by the vertex sequence $v^C_0$, $v^C_1$, \dots, $v^C_5$.
By Lemma~\ref{l6}, there exists an $T$-modification of $\phi'$ such that one of 
the conditions from Figure~\ref{fig10} is satisfied. We set $\phi'$ to be this $T$-modification.

\item We delete the edges of $T$ and add three edges according to the satisfied condition (if more than one condition is satisfied, we choose one of the conditions arbitrarily): 
For all three equalities $a_i=a_j$, where $i<j$, the satisfied condition requires,
we add an edge $e_i^T$ between $v^T_i$ and $v^T_j$ to $H$ and set $\phi'(e^T_i):=a_i$.
After we do this for all three equalities $\phi'$ is a flow on $H$.
Thanks to this substitution the satisfied condition from Figure~\ref{fig10} remains satisfied till the end of the procedure for $\{C, C^*\}$.
\end{enumerate}

From now on we stop modifying the flow values on edges currently present in the graph $H$. 

\item For each $6$-circuit $C$ of $F$ we perform the following subprocedure.
At the start of each subprocedure $\phi'$ is a flow on $H$. 
\begin{enumerate}
\item For all three equalities $a_i=a_j$, where $i<j$, of the fixed condition satisfied for $C$
we remove the edge $e_i^C$ and add the edges from $E(C)$ into $H$. 
We set $\phi'(e):=0$ for all $e\in E(C)$.
Now $\phi'$ is a $C$-extensible chain such that
one of the conditions from Figure~\ref{fig6} is satisfied.

\item We change $\phi'$ into a flow by modifying flow values of the edges from $E(C)$.
As $C$ has no multiedges, 
we denote the edges of $C$ here by two vertices the edge is incident with.
We may without loss of generality assume that the satisfied condition is either Condition 1, 3, or 6.
If for some $x\in\{\cR, \cG, \cB\}$ $\partial_G(C) \cap \phi'^{-1}(x) = \emptyset$, 
then we get a flow such that $|E(C)\cap Z| \le 1$ using statement one of Lemma~\ref{lfan}.
Thus we may assume that all three non-zero flow values are present on $\partial(C)$ in $\phi'$.
Assume Condition 1 is satisfied.
we may without loss of generality assume, that $a_0=a_1=\cR$, $a_2=a_3=\cG$, and $a_4=a_5=\cB$.
We set the flow values in $\phi'$ of $v_0v_1$, $v_1v_2$, $v_2v_3$, $v_3v_4$, $v_4v_5$, and $v_5v_0$ to 
B, G, $0$, G, R, and G, respectively.
Assume Condition 3 is satisfied.
we may without loss of generality assume, that $a_0=a_3=\cR$, $a_1=a_2=\cG$, and $a_4=a_5=\cB$.
We set the flow values in $\phi'$ of $v_0v_1$, $v_1v_2$, $v_2v_3$, $v_3v_4$, $v_4v_5$, and $v_5v_0$ to 
B, R, B, G, R, and G , respectively.
Assume Condition 6 is satisfied.
we may without loss of generality assume, that $a_0=a_3=\cR$, $a_1=a_4=\cG$, and $a_2=a_5=\cB$.
We set the flow values in $\phi'$ of $v_0v_1$, $v_1v_2$, $v_2v_3$, $v_3v_4$, $v_4v_5$, and $v_5v_0$ to 
B, R, G, B, R, and G, respectively. 
In each case Property~4 is satisfied for $C$ and this remains the case until the end of the procedure.
\end{enumerate}

Now $\phi'$ satisfies Property~4 and this remains true until the end of the procedure.

\item For each $\{C, C^*\}\in {\cal P}_G(F, >_1)$ we perform the following subprocedure.
Let $T=s(\{\C, C^*\})$.
At the start of each subprocedure $\phi'$ is a flow on $H$. 
\begin{enumerate}
\item For all three equalities $a_i=a_j$, where $i<j$, of the fixed condition satisfied for $T$
we remove the edge $e_i^T$ and add the edges from $E(T)$ into $H$. 
We set $\phi'(e):=0$ for all $e\in E(T)$.
Now $\phi'$ is a $T$-extensible chain
such that one of the conditions from Figure~\ref{fig10} is satisfied.

We may without loss of generality assume that the satisfied condition is either Condition 1, 3, 5, or 7.

\item
The following tabular contains, up to the isomorphism of $H$ and the isomorphisms of $\Z_2 \times \Z_2$,
all possible flow values on $\partial_G(T)$. 
we start enumeratinf the flow values by starting with flow values satisfying Condition~1,
then we continue with flow values satisfying Condition~3,~5, and~7.
We set the flow values of flow on $w^T_0w^T_1$ and $w^T_2w^T_3$
to the values indicated in the tabular. 
\begin{center}
\begin{tabular}{|c|c|c|c|| c |}
\hline
No.&R&G&B & flow on $w^T_0w^T_1$ and $w^T_2w^T_3$\\
\hline
1, 3, 5, 7&$a_0, \dots, a_5$ &&& G, B\\
1&$a_0, a_1, a_2, a_3$ &  $a_4, a_5$ &&G, B\\
1, 7&$a_0, a_1, a_4, a_5$ & $a_2, a_3$ && B, R\\
1, 3&$a_2, a_3, a_4, a_5$ & $a_0, a_1$ && B, G\\
1&$a_0, a_1$ & $a_2, a_3$ & $a_4,a_5$ & \bf G, B\\

3&$a_0, a_1, a_2, a_4$ & $a_3, a_5$ && G, B\\
3, 5&$a_0, a_1, a_3,a_5$ & $a_2, a_4$ && B, G\\
3&$a_0, a_1$ & $a_2, a_4$ & $a_3, a_5$ & \bf B, R\\

5, 7&$a_0, a_2, a_3, a_4$ & $a_1, a_5$ && B, R\\
5&$a_1, a_2, a_4, a_5$ & $a_0, a_3$ && R, B\\
5&$a_0, a_3$ & $a_1, a_5$ & $a_2, a_4$ & \bf B, B\\

7&$a_1, a_2 a_3, a_5$ & $a_0, a_4$ && R, B\\
7&$a_0, a_4$ & $a_1, a_5$ & $a_2, a_3$& \bf R, R\\ 
\hline
\end{tabular}
\end{center}

Without loss of generality assume that $C$ contains $v_0$ and $C^*$ contains $v_3$.

\item $H:=H/C^*$

\item Lemma~\ref{lemma5c} allows us to extend the $H$-extensible chain $\phi'$ 
into a flow without introducing zero values. We set $\phi'$ to be this flow.

\item $H:=H \text{ uncontract }C$.

\item 
If $\partial(T)$ contains all non-zero flow values, then we use Lemma~\ref{lemma5c} to extend the $C^*$-extensible chain $\phi '$ into a flow without introducing zero values (these cases are distinguished by bold flow values in the last column of the tabular above.)
Otherwise, use the first statement of Lemma~\ref{lfan} to extend the chain $\phi'$ without introducing more than one zero.
In both cases Property~5 is satisfied for $S$ and this remains the case until the end of the procedure.
\end{enumerate}

Now $\phi'$ satisfies Property~5. It is a flow and it satisfies Properties~1--4 as well. Thus $\phi'$ is good.
\end{enumerate}
Now let $\phi$ denote the flow $\phi'$ with range reduced to $E(G)$. As $\phi'$ was good, $\phi$ is good.
\end{proof}

\section{The cover}

Let $F$ be a $2$-factor of $G$ such that $G/F$ is $5$-odd-edge-connected,
let $>_1$ be an ordering of the circuits in $\C^I_G(F)$ and let $\phi$ be a flow on $G$.
We assign the circuits of $F$ into \emph{types} as follows. 
Circuit $C$ of $F$ is of
\begin{itemize}
\item \emph{Type t}, for $t\in \mathbb{N}-\{1, 3, 5\}$, if it has length $t$;
\item \emph{Type $N$}, if it is $\C_G^N(F)$;
\item \emph{Type $U$}, if it is from $\C_G^U(F,>_1)$;
\item \emph{Type $P_1$}, if it is from $\C_G^P(F,>_1)$ and $Z(\phi) \cap E(C))=\emptyset$;
\item \emph{Type $P_2$}: if it is from $\C_G^P(F,>_1)$ and $Z(\phi) \cap E(C)) \neq \emptyset$.
\end{itemize}
Note that as $G/F$ is $5$-odd-edge-connected, we have no circuits of length $1$ or $3$ in $F$
thus each circuit of $F$ has a type uniquely defined. Let ${\cal T}=(\mathbb{N} - \{1, 3, 5\}) \cup \{N, U, P_1, P_2\}$. 
For an circuit type $t\in {\cal T}$ we define $d_t(F, >_1, \phi)$ as the number of circuits of Type $t$ in $F$,
and $n_t$ as the number of vertices in one circuit of Type $t$
(circuits of a given type contain the same number of vertices).

\medskip

In the following two lemmas we use the same proof structure.
We define a cycle cover of $G$ based on a $2$-factor $F$, ordering $>_1$, function $s$, and a flow $\phi$
that is good with respect to $F$, $>_1$, and $s$. 
We bound the cover length of both covers using a discharging argument. 
In the \emph{starting charge} the edges of $G$ have charge according 
to how many times they are used in the cover. 
Then we define \emph{discharging rules} that may not decrease the charge.
After applying all the discharging rules we get the \emph{resulting charge}, 
where all the charge is assigned to the circuits of $F$.
For each circuit type we bound the resulting charge on a circuit of that type.  
This gives us an lower bound on the length of the cycle cover depending 
on the numbers and types of circuits in $F$.

\begin{lemma}\label{lemmacov1}
Let $G$ be a bridgeless cubic graph.
Let $F$ be a $2$-factor of $G$ such that $G/F$ is $5$-odd-edge-connected,
and let $>_1$ be an ordering of the circuits in $\C_G^I(F)$.
Let $s$ be an arbitrary $S$-assigning function and 
let $\phi$ be a flow on $G$ that is good with respect to $F$, $>_1$, and $s$.
Then there exists a cycle cover of $G$ of length at most
\begin{eqnarray*}
\sum_{t\in{\cal T}} a_t \cdot d_t(F, >_1, \phi),
\end{eqnarray*}
where the coefficients $a_t$ are given in Figure~\ref{figCovers}.
\end{lemma}
\begin{proof}
We define a cover that consists of three cycles. Each cycle is defined by selecting two values from $\Z_2\times Z_2$ for edges outside $F$ and two values from $\Z_2\times Z_2$ for edges in $F$. 
The edges having these values in $\phi$ are in the respective cycle.
\begin{center}
\begin{tabular}{c|c|c}
Cycle &$E(G)-E(F)$ & $E(F)$ \\
\hline
$C_1$&R, G & B, $0$\\
$C_2$&R, B & G, $0$\\
$C_3$&G, B & R, $0$
\end{tabular}
\end{center}
First, note that as $\phi$ is good (Property~1), the defined subgraphs cover all edges of $G$.
Second, the defined subgraphs are cycles as $C_i$, for $i\in\{1, 2, 3\}$, 
is the symmetric difference of the subgraph containing R and G edges for $C_1$, R and B edges for $C_2$, and
G and B edges for $C_3$ (which are cycles by Lemma~\ref{parity}) and $F$.

The discharging rules are the following.
The edges in $F$ send their charge to the circuit of $F$ they are contained in.
The edges outside $F$ send half of their charge to the circuit of $F$ that contains one of its endvertices
and  half of the charge to the circuit of $F$ that contains the second endvertex.

We bound the resulting charge of the circuits of $F$. 
For circuits $C$ of of length $i$ we bound the resulting charge of $C$ as follows.
The edges in $C$ are used once in the cover, except for the edges in $Z(\phi)$ which are used three times.
As $\phi$ is good (Property 2), at most quarter of the edges of $C$ are in in $Z(\phi)$.
Thus the charge sent to $C$ by the edges of $C$ is at most $i+2\lfloor i/4 \rfloor$.
Each edge in $\partial_G(C)$ is contained twice in the cover and it sends the charge $1$ to $C$ 
(if an edge is twice in $\partial_G(C)$ it sends the charge $1$ twice to $C$).
Altogether, the charge sent to $C$ is at most $2i+2\lfloor i/4 \rfloor$. 
If $C$ is of one of the following types, we can obtain an improved bound on the resulting charge of $C$:
\begin{itemize}
\item \underline{Type $4$, $8$, and $12$}: 
As $\phi$ is good (Property 2), for a circuit $C$ of this type $|E(C) \cap Z(\phi)<|E(C)/4|$. 
Thus the resulting charge of $C$ is at most $8$, $18$, and $28$, respectively.
\item \underline{Type $U$}: 
As $\phi$ is good (Property 3), for a circuit $C$ of this type 
$E(C) \cap Z(\phi)=\emptyset$. Thus resulting charge of $C$ is $10$.
\item \underline{Type $P_1$}: By definition of Type $P_1$, 
for a circuit $C$ of this type 
$E(C) \cap Z(\phi)=\emptyset$. Thus the resulting charge of $C$ is $10$.
\end{itemize}
From the above bounds we obtain that the defined cover satisfies the conditions of this lemma.
\end{proof}

\begin{figure}
\begin{center}
\begin{tabular}{c|c|c|c|c}
Type $t$ & $a_t$ & $b_t$ & $c_t = (a_t+2b_t)/2$ & $c_t/n_t$\\
\hline
Type $2$ & $4$ & $5$ & $4+2/3$ & $2+1/3$\\
Type $4$ & $8$ & $10$ & $9+1/3$ & $2+1/3$\\
Type $N$  & $12$ & $12+1/2$ & $12+1/3$ & $2+7/15$\\
Type $U$ & $10$ & $12+1/2$ & $11+2/3$ & $2+1/3$ \\
Type $P_1$ & $10$ & $12+1/2$ & $11+2/3$ & $2+1/3$\\
Type $P_2$ & $12$ & $11+1/2$ & $11+2/3$ & $2+1/3$\\
Type $6$ & $14$ & $14$ & $14$ & $2+1/3$\\
Type $7$ & $16$ & $16.5$ & $16+1/3$ & $2+1/3$\\
Type $8$ & $18$ & $19$ & $18+2/3$ & $2+1/3$\\
Type $9$ & $22$ & $20.5$ & $21$ & $2+1/3$\\
Type $10$ & $24$ & $23$ & $23+1/3$ & $2+1/3$\\
Type $11$ & $26$ & $24.5$ & $25$ & $<2+1/3$\\
Type $12$ & $28$ & $27$ & $27+1/3$ & $<2+1/3$\\
Type $i$, $i>12$ & $2i+2\lfloor i/4 \rfloor$ & $3/2\cdot i+\lfloor i/2 \rfloor +3$ & 
$(a_t+2b_t)/3$ & $<2+1/3$
\end{tabular}
\caption{The coefficients from the bounds from Lemma~\ref{lemmacov1}, ~\ref{lemmacov2}, ~\ref{lemmacov3}, and \ref{lemmamain}} \label{figCovers}
\end{center}
\end{figure}

\begin{lemma}\label{lemmacov2}
Let $G$ be a bridgeless cubic graph.
Let $F$ be a $2$-factor of $G$ such that $G/F$ is $5$-odd-edge-connected,
and let $>_1$ be an ordering of the circuits in $\C^I_G(F)$.
Let $s$ be an arbitrary $S$-assigning function and 
let $\phi$ be a flow on $G$ that is good with respect to $F$, $>_1$, and $s$.
Then there exists a cycle cover of $G$ of length at most
\begin{eqnarray*}
\sum_{t\in{\cal T}} b_t \cdot d_t(F, >_1, \phi),
\end{eqnarray*}
where the coefficients $b_t$ are given in Figure~\ref{figCovers}.
\end{lemma}
\begin{proof}
We start by manipulating $\phi$. 
First, we make a random permutation of non-zero elements of $\Z_2 \times \Z_2$ in $\phi$ 
(each permutation is chosen with probability $1/6$). 
We name the resulting flow $\phi_1$. 

We define $\phi_2$ by the following procedure that uses two variables: a chain $\phi'$ and a graph $H$. 
\begin{enumerate}
\item We set $\phi':=\phi_1$ and $H:=G/F$.

\item While changes to $\phi'$ are possible we do the following subprocedure. At the start of each subprocedure $\phi'$ is a flow on $H$. In this subprocedure we modify only the flow values on the edges from $E(G)-E(F)$.

\begin{enumerate}
\item If there is a circuit $C$ of $H$ consisting only of edges $e$ with $\phi'(e)=\cR$ in $H$, 
then we set $\phi'(e):=\cB$ for each $e\in E(C)$ (if there is more than one choice for $C$ we pick $C$ arbitrarily). 
As for each vertex even number of edges change their flow value from R to B, the sum at each vertex remains zero in $H$, thus $\phi'$ remains a flow.

\item If there is a circuit $C$ of $H$  consisting only of edges $e$ with $\phi'(e)=\cG$ in $H$, 
then we set $\phi'(e):=\cB$ for each $e\in E(C)$ (if there is more than one choice for $C$ we pick $C$ arbitrarily). 
As for each vertex even number of edges change their flow value from G to B, the sum at each vertex remains zero in $H$, thus $\phi'$ remains a flow on $H$.
\end{enumerate}

\item For each circuit $C$ of $F$ we do the following subprocedure. At the start of each subprocedure $\phi'$ is a flow on $H$.
In this subprocedure we modify only the flow values on the edges from $E(F)$.

\begin{enumerate}
\item  $H:=H \text{ uncontract } C$ and we set $\phi'(e):=\phi_1(e)$ for all $e\in E(C)$

The chain $\phi'$ may not be a flow on $H$ but it is a $C$-extensible chain. 

\item If $\phi'$ is not a flow (if this is the case then $\phi'(e) \neq \phi_1(e)$ for some $e\in \partial_G(C)$),
then we use Lemma~\ref{extend} to make a $C$-extension of $\phi'$ and we set $\phi'$ to be this extension. 
\end{enumerate}
\end{enumerate}

We name the resulting flow $\phi'$ on $G$ as $\phi_2$.
One can easily verify that $\phi_2$ has the following properties.

\medskip

\noindent {\it Property 1$'$}: For each $e\in E(G)-E(F)$, we have $\phi_2(e)\neq 0$.

\medskip

\noindent {\it Property 2$'$}: We have that $\phi_2^{-1}(\cR) - E(F) \subseteq \phi_1^{-1}(\cR) - E(F)$ and 
$\phi_2^{-1}(\cG) - E(F) \subseteq \phi_1^{-1}(\cG) - E(F)$.

\medskip

\noindent {\it Property 3$'$}: The edges $e$ with $\phi_2(e)=\cR$ form a forest in $G/F$.

\medskip

\noindent {\it Property 4$'$}: The edges $e$ with $\phi_2(e)=\cG$ form a forest in $G/F$.

\medskip

\noindent {\it Property 5$'$}: If for some $C$ $\phi_2(e)=\phi_1(e)$ for each $e\in\partial_G(C)$, then $\phi_2(e)=\phi_1(e)$ for each $e\in E(G)$.

\medskip

We define a cover that consists of three cycles. First two cycles are defined by selecting two values from $\Z_2\times Z_2$ for edges outside $F$ and two values from $\Z_2\times Z_2$ for edges in $F$. 
The edges having these values in $\phi_2$ are in the respective cycle.
The third cycle is defined by selecting two values from $\Z_2\times Z_2$ for edges outside $F$,
the edges having these values in $\phi_2$ are in the cycle, 
and for each circuit of $F$ we separately decide, 
whether we take the edges with flow values R and B in $\phi_2$ or the edges with flow values G and $0$ in $\phi_2$;
for each circuit we choose the possibility that contains fewer edges.

\begin{center}
\begin{tabular}{c|c|c}
Cycle &$E(G)-E(F)$ & $E(F)$ \\
\hline
$C_1$&R, G & B, $0$\\
$C_2$&R, G & R, G\\
$C_3$&R, B & R, B or G, $0$ (chosen for each circuit separately)
\end{tabular}
\end{center}
First, note that due to Property~1$'$ of $\phi_2$, the defined subgraphs cover all edges of $G$.
Second, the defined subgraphs are cycles as $C_i$, for $i\in\{1, 2, 3\}$, 
is the symmetric difference of
the subgraph containing R and G edges, R and G edges, and R and B edges 
(which are cycles by Lemma~\ref{parity}) and $F$, empty graph, and the union of circuits for which we chose
the edges with flow values G and $0$, respectively.

\medskip

The discharging rules are the following.
The edges in $F$ send their charge to the circuit of $F$ they are contained in.
The edges outside $F$ send  charge $1/2$ to circuit the  of $F$ that contains its endvertex,
and charge $1/2$ to the circuit of $F$ that contains the second endvertex.
This leaves charge $2$ on edges $e$ outside $F$ with $\phi_2(e)=\cR$ and 
            charge $1$ on edges $e$ outside $F$ with $\phi_2(e)=\cG$.

Finally, we add two \emph{global discharging rules}. 
As the \emph{first global rule},
we remove charges from the edges $e$ outside $F$ with $\phi_2(e)=\cR$ 
and we add charge $2$ to each circuit of $F$, 
except 
\begin{itemize}
\item circuits $C$ of length $2$, $4$, and $6$-circuit 
such that $\partial(C) \cap \phi_2^{-1}(\cR)=\emptyset$,  and 
\item circuits $C$ of Type $P_2$ paired to a circuit $C^*$
such that $Z(\phi_2) \cap E(C) \neq \emptyset$ and 
$\partial(s(\{C, C^*\})) \cap \phi_2^{-1}(\cR)=\emptyset$. 
\end{itemize}
The \emph{second global rule} the same as the first discharging rule, but all occurrences of $\cR$ is replaced by $\cG$ and the charge added to the circuits of $F$ is $1$ instead of $2$.

We prove that the global discharging rules may not decrease the charge. 
We prove it only for the first discharging rule, the same argument holds for the second rule
(we just use Property~4$'$ instead of Property~3$'$ of $\phi_2$). 
Let $H$ be a forest in $G/F$ formed by edges $e$ in $\phi_2(e)=\cR$ (Property~3$'$ of $\phi_2$). 
We removed charge $2|E(H)|$ from these edges. 
We can bound $|E(H)|$ to be at most
the number of vertices in $H$ minus the number of isolated vertices
in $H$ minus the number of isolated edges in $H$.
This is at most the number of circuits in $F$ minus the number of 
circuits $C$ such that $\partial(C) \cap \phi^{-1}_2(\cR)=\emptyset$
minus the number of pairs $\{C, C^*\}\in {\cal P}_G(F, >_1)$
such that $\partial(s(\{C, C^*\})) \cap \phi^{-1}_2(\cR)=\emptyset$.
Finally, as $\phi$ is good (Property~5) for a pair $\{C, C^*\}\in {\cal P}_G(F, >_1)$, $|Z(\phi) \cap (C \cup C^*)|\le 1$.
Thus $|E(H)|$ to be at most the number of circuits in $F$ minus the number of 
$2$, $4$, and $6$-circuits $C$ such that $\partial(C) \cap \phi^{-1}_2(\cR)=\emptyset$
minus the number of circuits $C$ from $\C^P(>_1,F)$ paired with a circuit $C^*$ 
such that $Z(\phi_2) \cap E(C) \neq \emptyset$ and $\partial(s(\{C, C^*\})) \cap \phi^{-1}_2(\cR)=\emptyset$.
As in the first global rule we add charge $2$ to exactly this many circuits of $F$,
the first discharging rule may not decrease the charge. 
Thus also the second discharging rule may not decrease charge.

\medskip

In general, for a circuit $C$ of length $i$ 
we bound the resulting charge of $C$ as follows. 
The edges in $C$ are either used once or twice in the cover, while at most half of the edges are used twice.
Thus the charge sent to $C$ by the edges of $C$ is at most $i+\lfloor i/2 \rfloor$.
The edges outside $C$, excluding the global discharging rules, send charge $i/2$ to $C$.
The first and the second global rule add charge at most $2$ and $1$ to $C$, respectively.
Thus, in general, we can bound the resulting charge of $C$ by $1.5i+\lfloor i/2 \rfloor+3$.
However, if $C$ is of one of the following types, then we can obtain 
an improved bound on the expected value of the resulting charge of $C$:
\begin{itemize}
\item \underline{Type $2$}:
By Lemma~\ref{parity}, $\partial(C)\cap \phi^{-1}(x) = \emptyset$ for two values $x \in \{\cR,\cG,\cB\}$. 
Thus $\partial(C)\cap \phi^{-1}_1(x) = \emptyset$ for each $x \in \{\cR,\cG,\cB\}$ with probability $2/3$.
By Property~2$'$ of $\phi_2$, 
$\partial(C)\cap \phi^{-1}_2(x) = \emptyset$ for each $x \in \{\cR,\cG\}$ with probability at least $2/3$.
Thus the expected value of the resulting charge of $C$ is 
at most $4$ before applying the global discharging rules, 
at most $4+1/3 \cdot 2 = 4+2/3$ after applying the first global rule, and
at most $4+2/3+1/3\cdot 1=5$ after applying the second global rule.

\item \underline{Type $4$}:
By Lemma~\ref{parity}, $\partial(C)\cap \phi^{-1}(x) = \emptyset$ for at least one value $x \in \{\cR,\cG,\cB\}$. 
Thus $\partial(C)\cap \phi^{-1}_1(x) = \emptyset$ for each $x \in \{\cR,\cG,\cB\}$ with probability at least $1/3$.
By Property~2$'$ of $\phi_2$, 
$\partial(C)\cap \phi^{-1}_2(x) = \emptyset$ for each $x \in \{\cR,\cG\}$ with probability at least $1/3$.
Thus the expected value of the resulting charge of $C$ is 
at most $8$ before applying the global discharging rules, 
at most $8+2/3 \cdot 2 = 9+1/3$ after applying the first global rule, and
at most $9+1/3+2/3\cdot 1=10$ after applying the second global rule.

\item \underline{Type $P_2$}: Let $C^*$ be the circuit paired with $C$. 
By the definition of Type $P_2$, $Z(\phi)\cup E(C) \neq \emptyset$.
As $\phi$ is good (Property~5), 
$\partial(s(\{C, C^*\}))\cap \phi^{-1}(x) = \emptyset$ for at least one value $x \in \{\cR,\cG,\cB\}$.
Thus $\partial(s(\{C, C^*\}))\cap \phi^{-1}_1(x) = \emptyset$ for each $x \in \{\cR,\cG,\cB\}$ with probability at least 
$1/3$.
By Property~2$'$ of $\phi_2$, 
$\partial(s(\{C, C^*\}))\cap \phi^{-1}_2(x) = \emptyset$ for each $x \in \{\cR,\cG\}$ with probability at least $1/3$.
Thus the expected value of the resulting charge of $C$ is 
at most $9+1/2$ before applying the global discharging rules, 
at most $9+1/2+2/3\cdot 2= 10+5/6$ after applying the first global rule, and
at most $10+5/6+2/3\cdot 1= 11 + 1/2$ after applying the second global rule.

\item \underline{Type $6$}: 
Assume first that $\partial(C)\cap \phi^{-1}(x) = \emptyset$
for at least one value $x \in \{\cR,\cG,\cB\}$.
In this case the expected value of the resulting charge of $C$ is at most
$12+2/3\cdot 2+2/3\cdot 1=14$. 

On the other hand, assume that $\partial(C)\cap \phi^{-1}(x) \neq \emptyset$ for all $x \in \{\cR,\cG,\cB\}$.
As $\phi$ is good (Property~1) and by Lemma~\ref{parity}, 
$|\partial(C)\cap \phi^{-1}(x)| =2$ for all $x \in \{\cR,\cG,\cB\}$.
As $\phi$ is good (Property~4) 
$|E(C) \cap \phi^{-1}(x)| \equiv |E(C) \cap Z(\phi)| \ \  (\bmod \ 2)$, for all $x\in\{\cR, \cG, \cB\}$.
Thus also $|\partial(C)\cap \phi^{-1}_1(x)| =2$ for all $x \in \{\cR,\cG,\cB\}$
and $|E(C) \cap \phi^{-1}_1(x)| \equiv |E(C) \cap Z(\phi_1)| \ \  (\bmod \ 2)$, for all $x\in\{\cR, \cG, \cB\}$.

Assume that for some edge in $e\in\partial(C)$ $\phi_1(e)\neq\phi_2(e)$.
By Properties~1$'$ and~2$'$ of $\phi_2$, $|\partial(C)\cap\phi_2^{-1}(\cR)|<2$ or
$|\partial(C)\cap\phi_2^{-1}(\cG)|<2$. Thus by Lemma~\ref{parity},
$\partial(C)\cap\phi_2^{-1}(\cR)=\emptyset$ or $\partial(C)\cap\phi_2^{-1}(\cG)=\emptyset$.
If the first condition is true, then we can bound the resulting charge of $C$ by
$12+0+1=13$. If the second condition is true, then we can bound the resulting charge of $C$ by
$12+2+0=14$. 

On the other hand, assume that for all $e\in\partial(C)$ it holds that  $\phi_1(e)=\phi_2(e)$.
Then by Property~5$'$ of $\phi_2$, for all $e\in E(C)$ $\phi_1(e)=\phi_2(e)$.
And thus $|E(C) \cap \phi^{-1}_2(x)| \equiv |E(C) \cap Z(\phi_2)| \ \  (\bmod \ 2)$, for all $x\in\{\cR, \cG, \cB\}$.
But then the third cycle in the cover is guaranteed to intersect $C$ in even number of edges,
and that is at most $2$.
Thus the resulting charge of $C$ is at most $11$ before applying the global discharging rules 
and at most $14$ after applying the global discharging rules.

We conclude that the expected value of the resulting charge of a $6$-circuit $C$ is at most $14$.
\end{itemize}

This bounds the expected length of the cover according to the inequality from the lemma statement.
This implies that at least one of the randomly chosen permutations of non-zero elements of $\Z_2\times\Z_2$ produces a cover satisfying the lemma statement.
\end{proof}

\begin{lemma}\label{lemmacov3}
Let $G$ be a bridgeless cubic graph.
Let $F$ be a $2$-factor of $G$ such that $G/F$ is $5$-odd-edge-connected,
and let $>_1$ be an ordering of the circuits in $\C^I_G(F)$.
Let $s$ be an arbitrary $S$-assigning function and 
let $\phi$ be a flow on $G$ that is good with respect to $F$, $>_1$, and $s$.
Then there exists a cycle cover of $G$ of length at most
\begin{eqnarray*}
\sum_{t\in{\cal T}} c_t \cdot d_t(F, >_1, \phi),
\end{eqnarray*}
where the coefficients $c_t=(a_t+2b_t)/3$ are given in Figure~\ref{figCovers}.
\end{lemma}
\begin{proof}
Let ${\cal C}_1$ be the cover whose existence is guaranteed by Lemma~\ref{lemmacov1} and
let ${\cal C}_2$ be the cover whose existence is guaranteed by Lemma~\ref{lemmacov2}.
Let $l({\cal C}_1)$ and $l({\cal C}_2)$ denote the lengths of the covers. 
\begin{eqnarray*}
\min\{l({\cal C}_1),l({\cal C}_2)\}
&\le& 1/3 \cdot l({\cal C}_1) + 2/3 \cdot l({\cal C}_1) \\
&\le& 1/3 \cdot \sum_{t\in{\cal T}} a_t \cdot d_t(F, >_1, \phi) +
2/3 \cdot \sum_{t\in{\cal T}} b_t \cdot d_t(F, >_1, \phi) \\
&=& \sum_{t\in{\cal T}} c_t \cdot d_t(F, >_1, \phi).
\end{eqnarray*}
Thus the shorter of these two covers satisfies the lemma statement. 
\end{proof}

Finally, we bound the length of the cover from Lemma~\ref{lemmacov3} using only $|E(G)|$ and $|\C^N_G(F)|$.

\begin{lemma}\label{lemmamain}
Let $G$ be a bridgeless cubic graph. Let $F$ be a $2$-factor of $G$ such that $G/F$ is $5$-odd-edge-connected.
Then $G$ has a cycle cover of total length at most 
$14/9\cdot|E(G)|+2/3 \cdot |\C^N_G(F)|$.
\end{lemma}
\begin{proof}
Let $>_1$ be an arbitrary ordering of the circuits in $\C^I(F)$. Let $s$ be an arbitrary $S$-assigning function.
By Lemma~\ref{lflow}, there exists a flow that is good with respect to $F$, $>_1$, and $s$.
Let $\C$ be the cover that satisfies Lemma~\ref{lemmacov3} 
and let $l(\C)$ denote the length of $\C$. As each circuit of $F$ has its type
$$
|V(G)|= \sum_{t\in{\cal T}} n_t \cdot d_t(F, >_1, \phi)
$$
Thus (see Figure~\ref{figCovers} for the coefficients $c_t/n_t$)
\begin{eqnarray*}
l(\C) 
&\le& \sum_{t\in{\cal T}} c_t \cdot d_t(F, >_1, \phi) \\
&=& \sum_{t\in{\cal T}} (c_t/n_t) \cdot d_t(F, >_1, \phi)\cdot n_t \\
&\le& 2/15 \cdot d_N(F, >_1, \phi)\cdot n_N+\sum_{t\in{\cal T}} 7/3 \cdot d_t(F, >_1, \phi)\cdot n_t \\
&\le& 2/3 \cdot |\C^N(F)|+7/3 \sum_{t\in{\cal T}}  \cdot d_t(F, >_1, \phi)\cdot n_t \\
&=& 7/3 \cdot |V(G)|+2/3 \cdot |\C^N(F)|=14/9.|E(G)|+2/3 \cdot |\C^N(F)|.
\end{eqnarray*}

\end{proof}

\section{Bounding $|\C^N_G(F)|$}

To prove Theorems~\ref{thm:main} and \ref{thm:main2}, we show how to pick a $2$-factor of $G$ that does not contain many circuits from $\C^N_G$. The approach is a variation on Proposition~5 from \cite{oddness} and Lemma~2.8 from \cite{sccc}.

Let $G$ be a bridgeless cubic graph.
Let $m=|E(G)|$ and let $E(G)=\{e_1, e_2, \dots, e_m\}$. Each perfect matching $M$ of $G$ can be represented by its \emph{characteristic vector} ${\bf x}^M \in \mathbb{R}^m$ in which the $i$-th coordinate ${\bf x}^M(i)$ is $1$  if $e_i\in M$ and $0$ otherwise. The \emph{perfect matching polytope} of $G$ is the convex hull of the characteristic vectors of all perfect matchings in $G$. 
Edmonds~\cite{edmonds} characterized the perfect matching polytope of a graph $G$ with an even number of vertices.
A vector ${\bf x}$ belongs to the perfect matching polytope of $G$  if and only if ${\bf x}$ satisfies the following conditions.
\begin{enumerate}
\item ${\bf x}(i)  \ge 0$ for all $i\in\{1, \dots m\}$.
\item $\sum_{e_i\in\partial_G(v)} {\bf x}(i) = 1$, for all $v\in V(G)$.
\item $\sum_{e_i\in\partial_G(A)} {\bf x}(i) \ge 1$, for $A \subseteq V(H)$ with $|A|$ odd.
\end{enumerate}

The following lemma allows us to bound the number of $5$-circuits from $\C^N$ that end up in a conveniently chosen $2$-factor.
\begin{lemma}\label{pmp}
Let $G$ be a cubic bridgeless graph, let $m=|E(G)|$ and let $E(G)=\{e_1, \dots, e_m\}$. 
Let ${\bf y}$ be a point in the perfect matching polytope of $G$
such that for each set $A \subseteq V(G)$ with $|\partial_G(A)|=3$, we have
$\sum_{e_i\in\partial(A)} {\bf y}(i) = 1$. Let $\C$ be a subset of the set of all $5$-circuits in $G$.
Then $G$ has a $2$-factor such that 
\begin{itemize}
\item $G/F$ is $5$-odd-edge-connected and
\item the number of circuits from $\C$ that are in $F$ is at most 
$$
\sum_{C\in\C} \frac{\left(\sum_{e_i\in\partial_G(C)} {\bf y}(i)\right) -1}{4}.
$$
\end{itemize}
\end{lemma}
\begin{proof}
As ${\bf y}$ is in the perfect matching polytope,
there exists an integer $k$, a set of $k$ perfect matchings $\{M_1, \dots, M_k\}$ and 
a set of $k$ positive real numbers $\{\alpha_1, \dots, \alpha_k\}$ that sum to $1$, 
such that
\begin{eqnarray}
{\bf y} = \sum_{i=1}^k \alpha_i {\bf x}^{M_i}. \label{eq1}
\end{eqnarray}

Let $A \subseteq V(G)$ such that $\partial_G(A)=\{e_a, e_b, e_c\}$, where $1\le a<b<c\le m$.
As a $3$-edge-cut in a cubic graph separates odd number of vertices, for each perfect matching $M$, the sum 
${\bf x}^M(a)+{\bf x}^M(b)+{\bf x}^M(c)$  is either $1$ or~$3$.
Note that if we sum coordinates $a$, $b$ and $c$ in (\ref{eq1})
the left side is equal to $1$ due to the lemma assumptions
while the right side is at least one.
As the coefficients $\alpha_i$ are positive, the equality is attained if and only if ${\bf x}^{M_i}(a)+{\bf x}^{M_i}(b)+{\bf x}^{M_i}(c)=1$ 
for each  $i\in \{1, \dots, k\}$.
Thus for each $i\in \{1, \dots, k\}$, $M_i$ contains exactly one edge of each $3$-edge-cut,
and the complementary $2$-factor $F_i$ crosses all $3$-edge-cuts of $G$. As $G$ is bridgeless, this implies that $G/F_i$ is $5$-odd-edge-connected.

\smallskip

We define
$$
f({\bf z})= \sum_{C\in \C} \sum_{e_i\in \partial(C)} {{\bf z}(i)}
$$
Due to linearity of $f$ and by (\ref{eq1}), we have
$$
f({\bf y})  = \sum_{i=1}^k \alpha_i f({\bf x}^{M_i}).
$$
Assume for contradiction that $f({\bf x}^{M_i})>f({\bf y})$ for all $i\in\{1, \dots, k\}$. But then 
we get $f({\bf y})  > \sum_{i=1}^k \alpha_i f({\bf x})$ which, as the coefficients $\alpha_i$ sum to $1$, 
is a contradiction. Thus, for some $j\in \{1, \dots, k\}$, we have $f({\bf x}^{M_j}) \le f({\bf y})$.

The circuits of $\C$ have odd number of vertices. 
Thus for each $C \in \C$ the $\sum_{e_i\in \partial(C)} {\bf x}^{M_j}(i)$ equals either $1$, $3$ or $5$.
Let $F_j$ be the $2$-factor complementary to $M_j$. 
If $C \in \C$ is in $F_j$,
then  $\sum_{e_i\in \partial(C)} {\bf x}^{M_j}(i)=5$.
Let $c$ be the number of circuits from $\C$ that are in $F_j$. Then
$$
\sum_{C\in \C} \sum_{e_i\in \partial(C)} {{\bf y}(i)} 
= f({\bf y}) \ge f({\bf x}^{M_j})=\sum_{C\in \C} \sum_{e_i\in \partial(C)} x^{M_j}(i) \ge \left(|\C|-c\right) \cdot 1+ c \cdot 5.
$$
From this inequality we have that $c\le \sum_{C\in\C} [(\sum_{e_i\in\partial_G(C)} {\bf y}(i)) -1]/4$ and $F_j$ satisfies both conditions of the lemma.
\end{proof}

Finally, we are ready to prove the main results of this paper.
\begin{proof}[Proof of Theorem~\ref{thm:main}]
As the vector $(1/3, 1/3, \dots, 1/3)$ is in the perfect matching polytope of $G$, using Lemma~\ref{pmp} with
$\C=\C^N_G$, there is a $2$-factor $F$ of $G$ such that $G/F$ is $5$-odd-edge-connected
and $|\C^N_G(F)|\le 1/6 \cdot |\C^N_G|$.

By Lemma~\ref{lemmamain}, $G$ has a cycle cover of length at most 
\begin{eqnarray*}
14/9\cdot |E(G)|+2/3 \cdot |\C^N_G(F)| &\le& 14/9\cdot |E(G)|+1/9 \cdot |\C^N_G| \\
&\le&
14/9\cdot |E(G)|+1/45 \cdot |V(G)| 
= 212/135 \cdot |E(G)|.
\end{eqnarray*}
\end{proof}

\begin{proof}[Proof of Theorem~\ref{thm:main2}]
Let $m=|E(G)|$ and let $E(G)=\{e_1, \dots, e_m\}$. We construct a vector ${\bf z}$ that is in the perfect matching
polytope of $F$ using the following procedure with variable ${\bf y}$ that is a vector of length $m$. 
Note that as $G$ is cyclically $4$-edge-connected, to show that ${\bf y}$ satisfies Equalities~1~and~3 from the characterisation of the perfect matching polytope, it is sufficient to show that all coordinates of {\bf y} are at least $1/5$.
\begin{enumerate}
\item ${\bf y}:=(1/3, \dots, 1/3)$. The vector ${\bf y}$ belongs to the perfect matching polytope of $G$.
\item For each edge $e_i$ that does not belong to any circuit from $\C^N_G$ but both its endvertices
belong to circuits from $\C^N_G$ we do the following subprocedure. Note that the endvertices of $e_i$
belong to different circuits from $\C^N_G$ because $G$ is cyclically $4$-edge-connected.
After each subprocedure ${\bf y}$ belongs to the perfect matching polytope of $G$.
\begin{enumerate}
\item ${\bf y}(i):=1/5$. This is the only step where we modify the value of ${\bf y}(i)$.
Therefore, this step decreases the value of ${\bf y}(i)$ by $2/15$. 
\item Let $v$ be one endvertex of $e_i$. Let $C\in\C^N$ be the circuit that contains $v$.
Let us denote the edges of $C$ by $e_a$, $e_b$, $e_c$, $e_d$, $e_f$ in the consecutive order so that
$v$ is incident with $e_a$ and $e_f$. We set 
${\bf y}(j):={\bf y}(j)+1/15$, for $j\in \{a, c, f\}$, and
${\bf y}(j):={\bf y}(j)-1/15$, for $j\in \{b, d\}$. 
We repeat the same operation for the second endvertex of $e_i$. Now 
Equality 2~from the characterisation of the perfect matching polytope of $G$ is satisfied by ${\bf y}$. 

For any edge $e_k$ that is in a circuit from $\C^N$ there are only two choices for edge $e_i$ such that
${\bf y}(k)$ is decreased. Thus ${\bf y}(k)$ does not decrease below $1/5$. Thus also Equalities~1~and~3
are satisfied by ${\bf y}$ after this step.
\end{enumerate}
\end{enumerate}
We name the vector ${\bf y}$ created by this procedure as ${\bf z}$. The vector ${\bf z}$ is from the perfect matching polytope of $G$ and has the following property

\medskip

\noindent \emph{Property}: For each edge $e_i$ that does not belong to any circuit from $\C^N_G$:
\begin{itemize}
\item if both endvertices of $e_i$ belong to circuits from $\C^N_G$, then ${\bf z}(i)=1/5$.
\item otherwise, ${\bf z}(i)=1/3$.
\end{itemize}

\medskip

Let $a$ denote the number of edges of $G$ that are not in a circuit from $\C^N_G$
but both its endvertices belong to circuits from $\C^N_G$. 
These edges are present in $\partial(C)$ for two circuits $C\in \C^N$ and contribute $2/15$ less 
to $\sum_{e_i \in \partial_G(C) {\bf y}(i)}$ than edges with ${\bf y}(i)=1/3$.
As the vector ${\bf z}$ is in the perfect matching polytope, using Lemma~\ref{pmp} with
$\C=\C^N_G$, there is a $2$-factor $F$ of $G$ such that $G/F$ is $5$-odd-edge-connected
and 
$$
|\C^N_G(F)| \le \sum_{C\in \C^N_G} \frac{(\sum_{e_i \in \partial_G(C)} {\bf y}(i))-1}{4} \le 1/6 \cdot |\C^N_G| - 2 \cdot \frac{2/15}{4} \cdot a = 1/6 \cdot |\C^N| - \frac{1}{15} \cdot a.
$$
By Lemma~\ref{lemmamain}, $G$ has a cycle cover of length at most 
\begin{eqnarray*}
14/9\cdot |E(G)|+2/3 \cdot |\C^N_G(F)| &\le& 14/9\cdot |E(G)|+1/9 \cdot |\C^N| - 2/45 \cdot a.
\end{eqnarray*}

Except for $a$ edges the each edge from $\partial(C)$, where $C\in\C^N$,
has its second end in a vertex that is not in a circuit from $\C^N$. 
Thus $5\cdot |C^N_G|- 2a \le 3 \cdot (|V(G)|-5\cdot|\C^N_G|)$ and
$|C^N_G|\le(3\cdot |V(G)| + 2a)/20$.  The cycle cover has length at most
\begin{eqnarray*}
14/9\cdot |E(G)|+1/60 \cdot |V(G)| - 1/30 \cdot a \le 47/30 \cdot |E(G)|.
\end{eqnarray*}

\end{proof}

\section*{Acknowledgement}
The author was supported by the grants APVV-15-0220 and VEGA 1/0813/18.



\begin{thebibliography}{99}
\bibitem{AT} N. Alon, M. Tarsi: Covering multigraphs by simple circuits, SIAM J. Algebraic Discrete Methods 6 (1985), 345--350.

\bibitem{BJJ} J. C. Bermond, B. Jackson, F. Jaeger: Shortest coverings of graphs with cycles, J. Combin. Theory Ser. B 35 (1983), 297--308.

\bibitem{sccc} B. Candráková, R. Lukoťka: \emph{Short Cycle Covers on Cubic Graphs by Choosing a $2$-Factor}, 
SIAM J. Discrete Math. 30 (2016), 2086--2106. 

\bibitem{diestel} R. Diestel: \emph{Graph Theory}, Electronic Edition 2000,
Springer-Verlag New York, 2000.

\bibitem{edmonds} J. Edmonds: \emph{Maximum matching and a polyhedron with $(0, 1)$ vertices}, 
J. Res. Nat. Bur. Standards Sect B. 69 B (1965), 125--130.

\bibitem{fan} G. Fan: \emph{Integer $4$-flows and cycle covers},
Combinatorica 37 (2017), 1097--1112.

\bibitem{jaeger} F. Jaeger: Flows and generalized coloring theorems in graphs, J. Combin. Theory Ser. B 26 (1979), 205--216.

\bibitem{JRT} U. Jamshy, A. Raspaud, M. Tarsi: Short circuit covers for regular matroids with nowhere-zero 5-flow, J. Combin. Theory Ser. B 43 (1987), 354--357.

\bibitem{JT} U. Jamshy, M. Tarsi: Short cycle covers and the cycle double cover conjecture, J. Combin. Theory Ser. B 56 (1992), 197--204.

\bibitem{kaiser}
T. Kaiser, D. Kráľ, B. Lidický, P. Nejedlý, R. Šámal: Short Cycle Covers of Cubic Graphs and Graphs with Minimum Degree Three, SIAM J. Discrete Math. 24 (2010), 330--355.

\bibitem{oddness} R. Lukoťka, E. Máčajová, J. Mazák, M. Škoviera: \emph{Small snarks with large oddness},
Electronic J. Combin. 22 (2015), \#P1.51. 


\bibitem{raspaud}
A. Raspaud: \emph{Flots et couvertures par des cycles dans les graphes et les matro\" ides}, 
Thèse de 3 ème cycle, Université de Grenoble, 1985.


\bibitem{seymour}
P.D. Seymour: Sums of circuits, Graph Theory and related topics (J.A. Bondy and U.S.R Murty, eds.), Academic Press, New York (1979), 341--355.

\bibitem{szekeres}
G. Szekeres: Polyhedral decomposition of cubic graphs, Bulletin of the Australian Mathematical Society 8 (1973), 367--387.

\bibitem{CQbook} C. Q. Zhang: Integer flows and cycle covers of graphs, CRC, (1997).
\end{thebibliography}
\end{document}